\documentclass{amsart}
\usepackage{amsfonts}
\usepackage{amsmath}
\usepackage{amssymb}
\usepackage{mathrsfs}
\usepackage{enumitem}
\usepackage{slashed}
\usepackage{multirow}
\usepackage{mathtools}
\usepackage{graphicx, subcaption, float}
\usepackage[dvipsnames]{xcolor}
\usepackage[colorlinks=true,citecolor=blue]{hyperref}
\usepackage{pdfpages}
\usepackage{comment}

\usepackage{cite}

\newtheorem{theorem}{Theorem}[section]
\newtheorem{proposition}[theorem]{Proposition}
\newtheorem{lemma}[theorem]{Lemma}
\newtheorem{question}{Question}

\theoremstyle{definition}
\newtheorem{definition}[theorem]{Definition}

\theoremstyle{remark}
\newtheorem{remark}[theorem]{Remark}
\newtheorem{hypo}{Hypothesis}
\numberwithin{equation}{section}
\newcommand\numberthis{\addtocounter{equation}{1}\tag{\theequation}}
\graphicspath{{Figures/}}
\newcommand{\rev}[1]{#1}

\newcommand{\grad}{\vec \nabla}

\newcommand{\bnat}{{b_{\star}}}
\newcommand{\uh}{\textbf{h}}
\newcommand{\uA}{\underline{A}}

\newcommand{\Peps}{\mathcal{P}_0^{\varepsilon}}

\newcommand{\Spec}{{\rm Spec}}

\newcommand{\Tper}{T_{\rm per}}
\newcommand{\Tpereps}{\Tper^\varepsilon}

\newcommand{\bx}{{\bf x}}
\newcommand{\bc}{{\bf c}}
\newcommand{\by}{{\bf y}}
\newcommand{\bxi}{\xi}

\newcommand{\bv}{{\bf v}}
\newcommand{\bk}{{\bf k}}
\newcommand{\bks}{{\bf k}_{\star}}

\newcommand{\bK}{{\bf K}}
\newcommand{\bm}{{\bf m}}
\newcommand{\BL}{{\rm BL}_{\varepsilon}}
\newcommand{\bn}{{\bf n}}

\newcommand{\Meff}{M_{\rm eff}^{\varepsilon}}

\newcommand{\R}{\mathbb{R}}
\newcommand{\C}{\mathbb{C}}

\begin{document}

\title[Near Invariance in Floquet Hamiltonians]{Near invariance of quasi-energy spectrum of  Floquet Hamiltonians}

\author[]{Amir Sagiv}
\address{Department of Mathematical Sciences, New Jersey Institute of Technology, Newark, NJ 07102, USA}
\email{amir.sagiv@njit.edu}
\author[]{Michael I.\ Weinstein}
\address{Department of Applied Physics and Applied Mathematics and Department of Mathematics, Columbia University, New York, NY 10027, USA}
\email{miw2103@columbia.edu}



\date{\today}

\begin{abstract}
The spectral analysis of the unitary monodromy operator, associated with a time-periodically (parametrically) forced Schr{\"o}dinger equation, is a question of longstanding interest.   Here, we 
 consider this question for Hamiltonians of the form
 $$H^{\varepsilon}(t)=H^0 + \varepsilon^a W(\varepsilon^a t, -i\nabla)\, ,$$ where $H^0$ is an unperturbed autonomous Hamiltonian, $a\geq 1$, and $W(T,\cdot)$ has a period of $\Tper >0$. 
 In particular, in the small $\varepsilon>0$ regime, we seek a comparison
  between the spectral properties of the monodromy operator,
  the one-period flow map associated with the $H^\varepsilon(t)$ dynamics, and that of the autonomous (unforced) flow, $\exp[-iH^0 \Tper \varepsilon ^{-a}]$. 
  We consider $H^0$ which is spatially periodic on $\R ^n$ with respect to a lattice.
  Using the decomposition of $H^0$ and $H^\varepsilon(t)$ into their actions on spaces (Floquet-Bloch fibers) of pseudo-periodic functions, we establish a  spectral near-invariance property for the monodromy operator, when  acting on data which are $\varepsilon$-localized in energy and quasi-momentum. Our analysis 
  requires the following steps: {\it (i)} spectrally-localized data are approximated by {\it band-limited (Floquet-Bloch) wavepackets}; {\it (ii)} the envelope dynamics of such wavepackets is well approximated by an effective (homogenized) PDE, and {\it (iii)} an exact invariance property for band-limited Floquet-Bloch wavepackets, which follows from the effective dynamics.  
We apply our general results to a number of periodic Hamiltonians, $H^0$, of interest in the study 
of photonic and quantum materials.   
\end{abstract}

\maketitle

\section{Introduction}

We consider a class of $n$-dimensional Schr{\"o}dinger equations with time-periodic forcing, governing
 $\psi=\psi(t,\bx)$,  a complex-valued function of $\bx \in \R ^n$ and $t\in\R$:
\begin{equation}\label{eq:lsa_bal}
i\partial_t \psi  = H^{\varepsilon}(t)\psi \, , \qquad H^{\varepsilon}(t) \equiv H^0 +
\varepsilon^a W(\varepsilon^a t,-i\nabla)  \ ,
\end{equation} 
where $a\geq 1$, \rev{$H^0$ is self adjoint, and $W(T,-i\nabla)$ and $H^{\varepsilon}(t)$ are self adjoint for all $t\geq 0$ and sufficiently smooth as functions of $t$.} 
Furthermore,  $T\mapsto  W(T, \cdot )$ is periodic of period $\Tper>0$, i.e., $W(T,\cdot)=W(T+\Tper. \cdot )$ for  all $T\in \R$. Hence, 
\begin{equation} t\mapsto H^\varepsilon(t)\quad \textrm{is periodic of period}~~\Tpereps \equiv\Tper\varepsilon^{-a} \, , \qquad  a\geq 1 \, .\label{Tper}
\end{equation}
We consider \eqref{eq:lsa_bal} for  $\varepsilon > 0$ and small: the regime of small and slowly
 varying time-periodic forcing. Very briefly,  wave-packet initial-data will deform
  on a time-scale which depends on its spectral localization. The parameter $a\geq 1$ is therefore chosen so that this time-scale and the forcing period $\Tpereps\sim \varepsilon^{-a}$  are matched; see Section~\ref{discuss-main}.

Since $H^{\varepsilon}(t)$ is time-dependent (non-autonomous), the spectra
of the family of operators $\{H^\varepsilon(t)\}_{t\in \R}$ does not determine the time-dynamics \eqref{eq:lsa_bal}.  Instead, one must study one-period evolution map associated with \eqref{eq:lsa_bal}:
for initial data
 $\psi(t)\big|_{t=0}=\psi_0\in L^2(\mathbb{R}^n)$, the solution $\psi(t)\in L^2(\R^n)$ of the initial value problem  \eqref{eq:lsa_bal}  is defined by the unitary operator 
 \begin{subequations}\label{eq:mono_def}
 \begin{equation}\label{eq:propagator} \psi^\varepsilon(t)=U^\varepsilon(t)\psi_0 \, .
 \end{equation}
The one-period evolution map, or {\it monodromy operator,} is the unitary operator
 \begin{equation}
 M^{\varepsilon} \equiv U^{\varepsilon} (\Tpereps): L^2(\R^n)\to L^2(\R^n) \, . \label{Meps}
 \end{equation}
 \end{subequations}
 
The driven and undriven problems can be compared by viewing  the autonomous case $W = 0$, i.e., the dynamics of  $ i\partial_t \psi = H^0 \psi$, 
as having (trivial) $\Tpereps$ periodicity. The associated solution operator is thus $ e^{-iH^0t}$ and the  monodromy operator is 
\begin{equation} M^{\varepsilon}_0=e^{-iH^0\Tpereps}\ .\label{M0eps}\end{equation} 
The spectrum of $M^{\varepsilon}_0$ acting in $L^2(\R^n)$ has a simple relation to the spectrum of $H^0$:
\begin{equation}\label{eq:quasienergy}  \Spec _{L^2(\R ^n)} (M^{\varepsilon }_0) = \left\{ e^{-iE  \Tpereps}  ~~ | ~~ E\in \Spec _{L^2 (\R ^n)}(H^0) \right\} .
\end{equation}
This relation motivates the notion of {\it quasi-energy}:
A point on the spectrum of the monodromy operator $z\in \Spec_{L^2 (\R ^n)} (M^{\varepsilon})\subseteq S^1$ can be written as $z=e^{-i\Tpereps \nu}$. The  phase $\Tpereps \nu$ is called a Floquet exponent, and  $\nu\in \R/2\pi\mathbb{Z}$ is called  a {\em quasi-energy}. We ask:
\begin{question}\label{q1}
What is the relation between the spectrum of the monodromy operator $M^{\varepsilon}$ and that of $M^{\varepsilon}_0$, arising from the non-trivial time-periodic forcing?
\end{question}

In general time-periodic settings, beyond being unitary, very little is known about the spectrum of the monodromy operator. Note in particular that, even though the parametric forcing is slow in time, the study of Question \ref{q1} is not covered by standard adiabatic theory, to the best of our knowledge; see discussion of Sec.\ \ref{sec:lit}.

In this paper, we gain insight on this question
  for the class of operators $M^{\varepsilon}$, where the unperturbed Hamiltonian, $H^0$, is periodic with respect to spatial translations in a lattice $\Lambda\subset\mathbb{R}_\bx^n$, i.e., $V(\bx+ \bv)=V(\bx)$ for all $\bv\in \Lambda$ and all $\bx\in\mathbb{R}^n$. Since $H^0$ commutes with $\Lambda$-translations, it may be decomposed into its action on distinct spaces of $\bk-$ pseudo-periodic functions, where $\bk$ varies over the {\it Brillouin zone}; see Section \ref{sec:FB}. This decomposition allows us to formulate and address a spectrally-local version of Question \ref{q1}:
 \begin{question}\label{q:piece_bal}
What is the relation between the spectrum of the monodromy operator $M^{\varepsilon}$ and that of $M^{\varepsilon}_0$, when restricted to $L^2(\R^n)$ data concentrated near an energy  $E_\star$ and quasi-momentum $\bks$? 
\end{question}

 \subsection{Discussion of main results}\label{discuss-main}
 We begin by discussing more specifically the type of spectral localization we have in mind.  First, fix a general quasi-momentum $\bk_\star$ and energy $E_\star$,  and consider initial data which is ``$\varepsilon-$  spectrally localized'' with respect to $H^0$-- a Bloch wave-packet of band-width $\varepsilon$ (see Proposition \ref{prop:BLproj}). The (unforced) evolution of such data by $U^0(t)=e^{-iH^0t}$, on large finite time-scales, has the structure of a slowly-varying spatial and temporal modulation
 of Bloch modes with energy $E_\star$ and quasi-momentum $\bk_\star$. 
The slow evolution
is described by an {\it effective Hamiltonian}, which is determined by 
  the local character  of the band structure (energy dispersion curves and eigenspaces) near $(\bk_\star,E_\star)$. The effective Hamiltonian captures the
  transport and spreading dynamics of such data. \rev{This work utilizes the existence of an effective Hamiltonian (and indeed, many known results, see Sec.\ \ref{sec:examples}) to answer Question \ref{q:piece_bal}.}
  
{\it We choose the forcing time-scale of our time-dependent Hamiltonian, $H^\varepsilon(t)$, so that there is  a non-trivial interplay between the unforced transport dynamics and the effect of time-dependent forcing.}
  \footnote{Such balancing corresponds to what is typically done in experiments. For example, a  wavepacket excitation is  designed by a choice of a laser frequency and pulse band-width, and balanced with  forcing to measure effects on experimentally accessible  spatial and temporal scales.}
  The parameter $a\geq 1$ in \eqref{eq:lsa_bal} is  chosen to achieve this balance of effects.
For example, the choice $a=1$ in \eqref{eq:lsa_bal} is appropriate for the dynamics
 where our Bloch wavepacket is concentrated near a point where the dispersion is locally linear, thus allowing for simple transport or conical diffraction, for example. In Section \ref{sec:examples}
  we discuss a number of examples.

A key to our analysis is  the space $\BL$ of band-limited wavepackets (Definition~\ref{def:BL}):  the space of Fourier band-limited envelope  modulations of Bloch modes ($\bk_\star$ pseudo-periodic eigenstates) of $H^0$ with energy $E_\star$.
 $\BL$ states are good approximations to $\varepsilon-$  spectrally localized a Floquet-Bloch wavepackets of band-width~$\varepsilon$ (Proposition \ref{prop:BLproj}).

Let $\Pi^\varepsilon$ denote the spectral \rev{projection-valued} measure associated with the unitary operator $M^{\varepsilon}$,  \rev{whose spectrum is on the unit circle $S^1$}
 (see Section \ref{sec:prelim});  let $\Peps$ be the projection onto the subspace of $L^2(\R^n)$, which is the space of functions which are $\varepsilon-$  localized, with respect to $H^0$, in quasi-momentum and energy about $(\bk_\star,E_\star)$. In particular, note that by applying the spectral theorem to $H_0$, then $\Peps M_0^\varepsilon= M_0^\varepsilon \Peps$.
  Our main {\it near-invariance} results are:
\begin{enumerate} 
\item  Theorem \ref{thm:Peps_e3}\ [Near invariance on ${\rm range}(\Peps)$]
Let $\mathcal{I}\subset S^1$ denote any arc such that the spectrum of $M^{\varepsilon}_0 \circ \Peps$ is contained in $\mathcal{I}$.\footnote{\rev{For linear operators with compatible domains we write composition as
$A\circ B$, so $(A\circ B)u := A(Bu)$.}} Then,  for $\varepsilon>0$ sufficiently small, 
$$ \Pi^{\varepsilon}\left[\mathcal{I}\right]\circ \Peps = \Peps  + \mathcal{O}_{B(L^2 (\R ^n))}(\varepsilon ^{n+1}) $$
or equivalently $ \Pi^{\varepsilon}\left[S^1 \setminus \mathcal{I}\right] \circ \Peps = \mathcal{O}_{B(L^2 (\R ^n))} (\varepsilon ^{n+1}) \, .$
 {\em This provides an answer to Question~\ref{q:piece_bal}:} when states in the range of  $\Peps$  evolve under the forced monodromy operator, $M^{\varepsilon}$, the resulting state has very small projection onto quasi-energies far from their quasi-energies under $M^{\varepsilon}_0=\exp(-iH_0 \Tpereps)$. Equivalently, denoting the spectral measure of $M^{\varepsilon}_0$ by $\Pi^{\varepsilon}_0$, 
 $$\left(\Pi^{\varepsilon}[\mathcal{I}]-\Pi^{\varepsilon}_0[\mathcal{I}]\right)\circ \Peps = \mathcal{O}_{B(L^2(\R^n))}(\varepsilon^{n+1}) \, .$$
 
 \end{enumerate}

\noindent 
Underlying the proof of Theorem \ref{thm:Peps_e3} is the following strict invariance result on $\BL$,  the space of band-limited wave-packets:
\begin{enumerate}[resume]
 \item 
  Theorem \ref{thm:main_bal}\ [Strict invariance on $\BL$] Let $\mathcal{I}$ be as in Theorem \ref{thm:Peps_e3}. Then, for all $\varepsilon>0$ and sufficiently small, we have the following {\em strict} invariance property for the evolution of $\BL$  under the flow of $M^{\varepsilon}$:
  $$\Pi^{\varepsilon}\left[\mathcal{I}\right] \circ {\rm Proj}_{\BL}  = {\rm Proj}_{\BL}, $$
or equivalently,  $\Pi^{\varepsilon}\left[S^1 \setminus \mathcal{I}\right] \circ {\rm Proj}_{\BL}  = 0$. 
\end{enumerate}

\begin{remark}\label{rem:O1effect} \rev{The results above are not captured by naive perturbative expansions of the propagator. Although the forcing term in (1.1) has magnitude \(\varepsilon^{a}\), the monodromy \(M^{\varepsilon}\) is defined by evolution over one period \(T_{\rm per}^{\varepsilon}\sim \varepsilon^{-a}\). Thus, an a priori estimate based solely on coefficient size and period does not guarantee a small correction to \(M^{\varepsilon}\). The relevant control arises from the existence of an effective Hamiltonian approximating the slow (envelope) dynamics. Our main theorem shows that the consequence such an approximation, on the spectral window of interest, is that the associated remainder is small in the sense pertinent to quasi-energies.}
\end{remark}
\bigskip

\subsection{Relevant analytical work on temporally-forced Hamiltonian PDEs}\label{sec:lit} 

In the present article, we give a novel perspective on the $L^2(\R^n)$ spectrum of parametrically forced Schr{\"o}dinger equations. \rev{Our approach is part of an ongoing interest in the wave-packet dynamics of non-autonomous (in time) and periodic (in space) Schr{\"o}dinger equations \cite{hameedi2022radiative, sagiv2022effective} and their associated non-autonomous homogenized dynamics \cite{bal2021multiscale, kraisler2025time}. This mathematical effort track the rapid development in the experimental study temporally-driven material, a category known as ``Floquet media.'' The techniques to implement this paradigm has been demonstrated in condensed-matter physics \cite{krieger1986time, perez2014floquet, wang2013observation}, photonics \cite{ozawa2019topological, Rechtsman-etal:13}, acoustics \cite{xue2022topological}, and more.}
Here, we draw connections and distinctions between our result and other approaches to similar problems. 
\subsubsection*{Reducibility} 
By translation invariance with respect to a lattice (see Section \ref{sec:FB}), the dynamics of \eqref{eq:lsa_bal} can reduce to the spectral study of the   {\em Floquet Hamiltonian} $$\mathcal{K}  \equiv i\partial_t - H^{\varepsilon} (t)  \, ,$$
over the family of spaces $\{ L^2 (S^1 ; L^2_{\bk}) \}_{\bk \in \mathcal{B}}$.
For each fixed $\bk$, spectral problems of this latter type correspond to time-periodically forced wave equations on the spatial torus. By constructing a change of variables which approximately maps 
 the original Hamiltonian to an autonomous Hamiltonian (reducibility),  it is shown that $\mathcal{K}(t)$ has pure point spectra, with quantitative control on the effect of the forcing on the unperturbed point spectrum  \cite{bambusi2001time, bambusi2017reduce, eliasson2008reducibility,  feola2020reducibility, howland1989floquetII, montalto2021linear}. For these results to hold, one needs strong assumptions regarding the growth of point eigenvalues of $H^0$. For Schr{\"o}dinger operators, the Weyl asymptotics imply that these growth assumptions are only satisfied for spatial dimension $n=1$. An exception to that is  \cite{eliasson2008reducibility},  which establish analogous results for $n\geq 2 $ for the case where $|V(x)|$ is assumed to be sufficiently small.  To the best of our knowledge, beyond these works, the nature of the spectrum of $\mathcal{K}$ remains an open problem~\cite{montalto2021linear}.

\subsubsection*{Adiabatic theory}
\rev{Adiabatic theory studies slowly varying Hamiltonians of the form
\begin{equation}\label{eq:adiabatic}
i\,\partial_t \psi^\varepsilon(t)=H(\varepsilon t)\,\psi^\varepsilon(t), \qquad 0<\varepsilon\ll1.
\end{equation}
Here $s\mapsto H(s)$ is a sufficiently regular curve of self-adjoint operators on a Hilbert space $\mathcal H$. A common hypothesis is the existence, for each $s$, of an isolated spectral subspace of $H(s)$ with associated projector $P(s)$, where the map $s\mapsto P(s)$ is smooth in an appropriate sense. Under such assumptions, if $\psi^\varepsilon(0)\in\operatorname{Ran}P(0)$, then for times $ t= O(\varepsilon^{-1})$ the solution remains close to $\operatorname{Ran}P(\varepsilon t)$; the error order depends on the regularity of $H$ (see, e.g., \cite{avron1987adiabatic,davies1978open,de2022locobatic,garrido1964generalized,hagedorn2002elementary, henheik2022adiabatic, joye2022adiabatic,kato1950adiabatic,nenciu1980adiabatic,nenciu1993linear}).

The conclusions of adiabatic theorems concern transport along moving spectral subspaces for $H(\varepsilon t)$. They do not, in general, yield statements about the \emph{spectrum of the unitary monodromy} associated with a time–periodic evolution. Our results address this different question (Question \ref{q:piece_bal}), and use a different mechanism -- the existence of an effective Hamiltonian $H_{\rm eff}$ governing the slow dynamics.}

\subsection{Outline of the paper.}\label{sec:outline}  In Section \ref{sec:prelim} we provide the necessary background on Floquet-Bloch theory of periodic Hamiltonians, the spectral theorem for unitary operators, and introduce relevant notation. Section \ref{Eb-loc} presents a brief intuitive introduction to effective dynamics and homogenization of periodic Hamiltonians. The main results of this article are presented in Section \ref{sec:main-results}. In Section \ref{sec:examples} we demonstrate how these results apply to a number of specific periodic Hamiltonians, $H^0$ and the associated $H^\varepsilon$, of physical interest. The proofs of the main results are presented in Section \ref{sec:pf_main}. Finally, a formal derivation of an effective transport equation (see Section \ref{sec:lin}) is presented in Section~\ref{sec:transport}.
\subsection{Acknowledgments}
AS would like to thank P.\ Kuchment and V.\ Rom-Kedar, whose questions inspired this research. 
This research was supported in part by National Science Foundation grants, DMS-2508811 (AS), 
DMS-1620418, DMS-1908657 and DMS-1937254 (MIW),  the Simons Foundation Math + X Investigator Award \#376319 (MIW, AS), the Binational Science Foundation grant \#2022254 (MIW, AS), and the AMS-Simons Travel Grant (AS).
\section{Mathematical Preliminaries}\label{sec:prelim}
 \subsection{Floquet-Bloch theory}\label{sec:FB}

We consider Hamiltonians $H^0=-\Delta+V$, where  $V$ is periodic with respect to a lattice $\Lambda=\mathbb{Z}\bv_1\oplus\cdots\oplus\mathbb{Z}\bv_n$, where  $\{\bv_1,\dots,\bv_n\}$ is a linearly independent set of vectors in $\mathbb{R}^n$.
\footnote{We believe our analysis can  be extended to general classes of elliptic operators (scalar Hamiltonians and systems)  $H^0$,
whose coefficients are periodic with respect to a lattice; see \cite{Kuchment:16}.}
 Since $H^0$ commutes with lattice translations,  the Hamiltonian $H^0$ admits a fiber-decomposition $H^0 =\int^\oplus_{\mathcal{B}} H_\bk^0 \, d\bk$, where $H^0_\bk$ denotes the operator
 $H^0$ acting in the space $L^2_\bk$, consisting of  $L^2_{\rm loc}$  functions which are $\bk-$ pseudoperiodic, {\it i.e.} $\psi\in L^2_\bk$ if $\psi(\bx+\bv)=e^{i\bk\cdot\bv}\psi(\bx)$ a.e. in $\bx\in\mathbb{R}^2$.  The set $\mathcal{B}$ is
  the Brillouin zone, a choice of fundamental cell in $(\mathbb{R}_\bx^n)^*=\mathbb{R}^n_\bk$. For each quasimomentum, $\bk \in \mathcal{B}$, $H_{\bk}^0$ is self-adjoint and has a compact resolvent. Hence, for each $\bk$,  $H_{\bk}^0$ has a real sequence of discrete eigenvalues of finite multiplicity
  $$
  E_1 (\bk)\le E_2(\bk)\le \dots\le E_b(\bk)\le\dots \, ,
$$ tending to infinity. The corresponding $L^2_\bk$ eigenfunctions, denoted by $\Phi_b (\bx; \bk)$: 
  \[ H^0 \Phi_b (\bx; \bk) = E_b(\bk) \Phi_b (\bx; \bk)\, ,\qquad  \bx\mapsto \Phi_b(\bx,\bk) \in L^2_\bk \, ,\]
  or equivalently $\bx\mapsto e^{-i\bk\cdot\bx}\Phi_b(\bx,\bk)\in L^2(\R^n/\Lambda)$.  
  These eigenfunctions may be chosen to form an orthonormal basis for $L^2_\bk$.  Each function $\bk \mapsto E_b (\bk)$ is continuous and piecewise analytic \cite[Theorem 5.5]{Kuchment:12}, and hence Lipschitz continuous. 
  
 Each image, $E_b(\mathcal{B})$, is a subinterval of $\R$ called  the $b^{\rm th}$ {\em spectral band.}
   The graphs of $\bk\mapsto E_b(\bk)$ are called  {\em dispersion surfaces}.  The collection of all pairs $(\bk,E_b(\bk))$ and corresponding normalized $L^2_\bk$ eigenfunctions, 
$\Phi_b (\bx; \bk)$, 
 is called the {\it band structure} of $H^0$. 
Finally,  the family of Floquet-Bloch modes $\cup_{\bk\in\mathcal{B}}\{\Phi_b(\cdot,\bk)\}_{b\ge1}$ is complete in $L^2(\R^n)$;
for any $f\in L^2(\R^n)$,
$$ f(\bx) = \frac{1}{{\rm vol}(\mathcal{B})}\sum_{b\ge1}\int_{\mathcal{B}} \left\langle \Phi_b(\cdot,\bk),f \right\rangle_{L^2(\R^n)} \Phi_b(\bx,\bk) d\bk \, , $$
where the sum is interpreted as a convergence of partial sums in $L^2(\R^n)$. For simplicity, we will assume henceforward that ${\rm vol}(\mathcal{B})=1$.
\subsection{$U^\varepsilon(t)$, $M^\varepsilon$ and its associated spectral measure}\label{sec:spectral_thm}

In  \eqref{eq:propagator} we introduced the (unitary in $L^2 (\R ^n)$) evolution $U^{\varepsilon}(t)$  associated with the dynamics \eqref{eq:lsa_bal}.  In this section we give a brief outline of a construction of $U^\varepsilon(t)$ and then discuss the spectral measure of the associated monodromy operator. Under very general assumptions on $H^0$ and the operators $\{W(t,\cdot)\}$, a unitary propagator can be shown to exist, \rev{see e.g., \cite{howland1974stationary, yajima1987existence}.}

From $U^{0}(t)=e^{-iH^0t}$ 
 and  $U^{\varepsilon}(t)$, we obtain the unitary monodromy operators  $M_0 ^{\varepsilon}$ and $M^{\varepsilon}$ which, by the spectral theorem, are equipped with associated  spectral (projection-valued) measures $\Pi^{\varepsilon}_0$ and $\Pi^{\varepsilon}$, respectively. For completeness, we review the definition and properties of a spectral measure and the spectral theorem for unitary operators. We refer the reader to \cite{davies1996spectral, hall2013quantum, RS4, taylor2013partial} for details.

Let $\mathcal{H}$ be a Hilbert space, let $X$ be a set, and $\Sigma$ a $\sigma$-algebra in $X$. A map $\Pi:\Sigma \to B(\mathcal{H})$, the Banach space of bounded linear operators on $\mathcal{H}$, is called a {\em projection-valued measure} if the following properties hold:
\begin{enumerate}
\item $\Pi(I)$ is an orthogonal projection for every $I\in \Sigma$.
\item $\Pi(\emptyset)=0$ and $\Pi(X) = {\rm Id}$.
\item If $\{I_j\}_{j\geq 1} \subset \Sigma$ are disjoint then $$\Pi \left( \bigcup\limits_{j\geq 1} I_j\right) v = \sum\limits_{j\geq 1} \Pi(I_j) v \, , \qquad v\in \mathcal{H} \, .$$
\item $\Pi (I_1 \cap I_2) = \Pi (I_1)\Pi(I_2)$ for all $I_1, I_2 \in \Sigma$.
\end{enumerate}
 \begin{theorem}
Let $U$ be a unitary operator on $\mathcal{H}$. There exists a unique projection-valued measure $\Pi = \Pi_U$ on the Borel $\sigma$-algebra of $S^1$, which contains the spectrum of $U$,  such that for every $f\in \mathcal{H}$
$$\int\limits_{S^1} z \, d\Pi(z) f = Uf \, .$$
\end{theorem}

\subsection{Notation and conventions}

\begin{enumerate}
\item We adopt the convention of considering all $\C^N$ vectors as column vectors.
If $a,b\in\C^N$, $a^\top b= a\cdot b$.
\item  Fourier transform on $L^2(\R^n)$: For a function, $f$, defined on $\R^n$, define its  Fourier transform as
\[
\hat{f}(\xi) =\mathcal{F}[f](\xi)=\hat{f}(\xi) \equiv \rev{\frac{1}{(2\pi)^{n/2}}}\int\limits_{\R ^n} f(x) e^{- i \xi \cdot x} \, d x. \, \]
  Furthermore, introduce
\[
\check{g}(x) \equiv \rev{\frac{1}{(2\pi)^{n/2}}}\int\limits_{\R ^n} g(\xi) e^{i x \cdot \xi} \, d\xi .
\]
The mappings $f\mapsto \hat{f}$ and $g\mapsto\check{g}$ map Schwartz class, $\mathcal{S}(\R^n)$, to itself and for all $f \in \mathcal{S}(\R^n)$, we have the Plancherel identity  $\|f\|_{L^2 (\R ^n)} = \|\hat{f}\|_{L^2 (\R ^n)}$ and the inversion formula $\left(\hat{f}\right)^{\check{}}=f$. Hence, $\check{f}=\mathcal{F}^{-1}f$ on 
$\mathcal{S}(\R^n)$.  By density, both mappings extend to bounded linear transformations on $L^2(\R^n)$ which satisfy the Plancherel identity, and the inversion formula.
\item Pauli matrices are given by $\sigma_0 = {\rm I}$, and 
$$ \sigma_1 = \left(\begin{array}{cc}
0 &1 \\ 1 &0
\end{array} \right) \, , \qquad \sigma_2 = \left(\begin{array}{cc}
0 &-i \\ i &0
\end{array} \right) \, ,\qquad \sigma_3 = \left(\begin{array}{cc}
1&0  \\ 0& -1
\end{array} \right) \, . $$
\item For a vector $\bv=(v_1,v_2)\in\C^2$, we write $(v_1,v_2)\cdot(\sigma_1,\sigma_2)=v_1\sigma_1+v_2\sigma_2$.
\end{enumerate}

\section{Wavepackets, the geometry of dispersion surfaces, and periodic homogenization}\label{Eb-loc}

Our spectrally local formulation concerning the quasi-energy spectrum $H^\varepsilon(t)$, Question  \ref{q:piece_bal}. is a natural relaxation of Question \ref{q1}.
 In physical settings, a crystalline structure is  experimentally probed in a narrow spectral range, {\it e.g.}
a bulk material is externally excited  (e.g. electrically, optically, elastically, acoustically). Such settings induce the propagation of spectrally localized wavepackets (quasi-particles), 
whose envelope dynamics are given  by a simplified effective Hamiltonian. 

To illustrate this last point and how effective Hamiltonians emerge, consider  the following ``toy model'' of continuously translation-invariant and time-periodically forced Hamiltonian dynamics governing a wave-field $\psi=\psi(t,\bx)$:
\begin{equation}
\begin{cases}
    i\partial_t\psi (t,\bx) &= E(-i\nabla)\psi \ + \ \varepsilon^a\rev{\mathcal{Q}}(\varepsilon^a t)\psi \\
    \psi(0,\bx) &=\psi_0(\bx) \, ,
    \end{cases}
 \label{toy}
\end{equation}
where $\psi_0$ is sufficiently smooth and localized on $\R^n$ \rev{and $\mathcal{Q}$ is a real integrable function}.
The real-valued {\it dispersion relation}  $\bxi \mapsto E(\bxi)$ is, for simplicity, taken to be smooth.
Clearly, an explicit solution can be given in terms of the Fourier transform, but our goal here will be to 
discuss the notion of effective dynamics. 
 
Consider initial data whose Fourier transform is concentrated near  $\bxi_{\star} \in \R ^n$: 
\[\widehat{\psi^\varepsilon_0}(\bxi) = \varepsilon^{-n}\widehat\Psi_0(\varepsilon^{-1}(\bxi-\bxi_{\star})) \, , \qquad \widehat\Psi_0\in\mathcal{S}(\R^n) \, ,\ 0<\varepsilon\ll1.\]
The solution of the initial value problem \eqref{toy} may be written as:
\[ \psi^\varepsilon(t, \bx) = \rev{\frac{1}{(2\pi)^{n/2}}}e^{i(\bxi_{\star}\cdot\bx-E(\bxi_{\star})t)} \int e^{i\left( [E(\bxi_{\star}+\varepsilon\tilde{\bxi})-E(\bxi_{\star})]t +\varepsilon\tilde{\bxi}\cdot\bx + \Theta(\varepsilon^a t) \right]}\widehat\Psi_0(\tilde{\bxi}) d\tilde{\bxi},
\]
where 
$ \Theta(T) \equiv \int_0^T \rev{\mathcal{Q}}(s) ds$.

 If $\nabla E(\bxi_{\star})\ne0$, then by Taylor expansion of $E(\bxi)$ about $\bxi_{\star}$, 
we obtain the following approximation of the solution $\psi_\varepsilon(t, \bx)$ of \eqref{toy} with $a=1$,  which is valid on the time scale:  $0\le t\lesssim \varepsilon^{-1}$:
\[\psi_\varepsilon(t, \bx) \approx  e^{i(\bxi_{\star}\cdot\bx-E(\bxi_{\star})t)} \cdot  B(\varepsilon t, \varepsilon\bx),\  
\]
where the envelope $B(T,X)$ is governed by a driven transport equation
\[ i\partial_T B(T,X) =  \left[\ i\nabla_{\bxi} E(\bxi_{\star}) \cdot \nabla_X \ +\ \rev{\mathcal{Q}(T)}\right]B(T,X) \, .\]

If, on the other hand, $\nabla E(\bxi_{\star})=0$ and $D_\bk^2 E(\bxi_{\star})$, the $n\times n$ Hessian matrix, is non-singular, then 
we obtain the following approximate solution $\psi_\varepsilon(t, \bx)$ of \eqref{toy} with $a=2$, which is valid on the time scale:  $0\le t\lesssim \varepsilon^{-2}$:
\[  \psi_\varepsilon(t, \bx) \approx   e^{i(\bxi_{\star}\cdot\bx-E(\bxi_{\star})t)} \cdot B(\varepsilon^2 t, \varepsilon\bx) \, , \]
where $B(T,X)$ satisfies an (generally anisotropic) effective Schr{\"o}dinger equation:
\[ i\partial_TB(T,X)  =  \left[\ \nabla_X\cdot \frac12 D_{\bxi}^2 E(\bxi _{\star}) \nabla_X\ +\ \rev{\mathcal{Q}(T)}\right]B(T,X) \, .\]

In each case, the function $B(T,X)$, which provides the slow envelope evolution,  is governed by a  time-dependent effective Hamiltonian: 
\begin{equation}
i\partial_T B(T,X) = H_{\rm eff}(-i\nabla,T) B(T,X),\label{eff-dyn}
\end{equation}
in which both the effects of deformation under $H^0$ and temporal forcing are captured. Note also
that  $H_{\rm eff}(-i\nabla,T)$  commutes with continuous spatial translations and therefore
can be analyzed using the Fourier transform. 

 In general, for  spatially homogeneous media and for the case of crystalline (lattice periodic) media described by $H^0$, which is invariant under discrete translations in a lattice, the dispersion relation eigenvalue-branches 
 may be degenerate. At such degeneracies the dispersion relations $\bk\mapsto E_b(\bk)$ may not be smooth,
 although they are Lipschitz continuous  if $H^0$ is self-adjoint. Furthermore, in such cases, the eigenvector maps $\bk\mapsto\Phi_b(\bx;\bk)$ may even be multivalued.  \footnote{
In this paper, we discuss only isolated point degeneracies. Other types of band degeneracies may arise.
Examples are (i) the touching of two bands along a submanifold of quasi-momenta due the underlying symmetries and (ii) degeneracies of infinite multiplicity such as ``flat bands,'' as in e.g., the Landau Hamiltonian \cite{HK:10}. We do not treat these situations in the present work.} Nevertheless, Fourier-type analysis
(based on Floquet-Bloch modes) and multiple-scale / homogenization methods can
  be used to rigorously derive, with accompanying error bounds,  effective envelope dynamics. Examples are
  \begin{enumerate}
\item[(i)] effective mass Schr{\"o}dinger equations \cite{allaire2005homogenization, hoefer2011defect} when $E_\star$ corresponding is at an isolated band edge, at which the dispersion surface is generically quadratic,
\item[(ii)] effective Dirac equations (with time-independent and time-dependent Hamiltonians) for dispersion surfaces touching conically  (Dirac points) \cite{FW:14, hameedi2022radiative, sagiv2022effective},
\item[(iii)]  effective matrix-Schr{\"o}dinger equations, for quadratically degenerate dispersion surfaces \cite{keller2018spectral}, and
\item[(iv)] effective Dirac operators of magnetic  type for non-uniform spatial deformations of honeycomb media \cite{GRW:21}
\end{enumerate}

\section{Main results}\label{sec:main-results}

\subsection{Hypotheses and definitions} 
  Our first assumption concerns the character of the energy band structure near $(\bk_\star,E_\star)$;
  in particular if  $(\bk_\star,E_\star)$ is a degeneracy, then this degeneracy is isolated:
  \begin{hypo}[Spectral separation]\label{hyp:deg} 
  Let $(\bk_\star,E_\star)$ be such that 
$H^0_{\bk_\star}$ has an eigenvalue $E_\star$ of multiplicity $N\ge1$, i.e., for some $\bnat \geq 1$
 \begin{equation}\label{eq:degeneracy}
   E_{\bnat -1} (\bks) < E_\star=E_{\bnat}(\bks)= E_{\bnat+1} (\bks)=\dots= E_{\bnat+N-1}(\bks)< E_{\bnat +N}(\bks) \, .
   \end{equation}
Furthermore,  $(\bks,E_\star)$ is isolated  in the band structure in the sense that
\[ E_{\bnat-1}(\bk) < E_{\star} < E_{\bnat+N}(\bk)\]
for all $\bk$ in an open neighborhood of the quasimomentum $\bks$.
Introduce an  \rev{orthonormal} basis for the degenerate eigenspace: 
\[ \{ \Phi_b(\bx,\bks) ~~ |~~ b_{\star}\leq b \leq b_{\star}+N-1 \} \, .\]
\end{hypo}
With Question \ref{q:piece_bal} in mind and assuming spectral separation as defined in Hypothesis \ref{hyp:deg}, we define a projection, $\Peps$,  associated with a subspace of $L^2(\R^n)$
consisting of states, which are superpositions of modes whose quasimomenta and energy  are near  $(\bks, E_{\star})$:
\begin{equation}\label{eq:Pdef}\Peps \equiv \int\limits_{|\bk-\bks|<\varepsilon} {\rm Proj}\left(\left|H^0_{\bk}-E_{\star}\right|< L \varepsilon\right) \, d\bk \, ,
\end{equation}
where $L>0$ is fixed. \rev{Here and henceforward, such integrals over $k$ are to be understood in the sense of the Floquet-Bloch deomposition of $L^2(\mathbb{R}^n)$.}

We next present two additional assumptions concerning the underlying wave-packet dynamics. 
Let $(\bks,E_\star)$ satisfy the spectral separation Hypothesis \ref{hyp:deg} 
with parameters $N\ge1$ and $\rev{b_\star}\ge1$. 
Denote the vector 
 of degenerate Floquet-Bloch modes:
\[ \Phi_\star(\bx) \equiv \left(\begin{array}{c}
\Phi_{b_\star} (\bx; \bks) \\ \vdots \\ \Phi_{\bnat+N-1} (\bx; \bks)  
\end{array} \right) \ .\]

We next introduce the subspace of $L^2(\R^n)$,
 consisting of Fourier band-limited wave-packets, which are  
modulations $\Phi_\star$. 
\begin{definition}[Band-limited wave-packets]\label{def:BL}
  For fixed parameters $\varepsilon , d_0 > 0$, we define:
\begin{align}
\label{eq:BL_def} \BL
 &\equiv \{ u=\alpha(\varepsilon \bx)^{\top} \Phi_\star(\bx)\ :\ {\rm supp}(\hat{\alpha})\subseteq B_{d_0}(0) ~~{\rm and} ~~ \alpha \in L^2 (\R ^n; \C ^N) ~ \} \,  , 
\end{align}
\end{definition}

\begin{hypo}[Translation invariant effective dynamics]\label{hyp:eff}
There is a one-parameter family of unitary operators on $\BL$, $U^{\varepsilon}_{\rm eff}(t)$,  with the following properties:
\begin{enumerate}
\item {\em (Spatially translation invariant effective dynamics)}\\
 For $\psi_0=\alpha_0 ^{\top}(\varepsilon \bx) \rev{\Phi_{\star}}(\bx)\in\BL $, $U^{\varepsilon}_{\rm eff}(t)$ is defined by:
 $$U^{\varepsilon}_{\rm eff}(t)\psi_0 = \frac{e^{-iE_\star t}}{\rev{(2\pi)^{n/2}}}\int\limits_{\R ^n} e^{ i\bxi \cdot \varepsilon \bx}\widehat{\mathscr{U}}_{\rm eff}(\varepsilon^a t;\bxi)\hat{\alpha}_0(\bxi)\, d\bxi \cdot \Phi_\star(\bx)\, . $$
where $(T,\xi)\mapsto \widehat{\mathscr{U}}_{\rm eff}(T;\bxi)$ is a smooth mapping from $\R_T\times \R_\xi^n$ into the space of unitary $N\times N$ matrices.
\item {\em (Approximation by effective dynamics)}\\ Let  $U^{\varepsilon}_{\rm eff}(t)$ be defined as in (1). 
If 
$\psi_0\in\BL $, then 
\begin{equation}\label{eq:abstract_validity}\lim\limits_{\varepsilon \to 0^+}  \sup_{0\le t\le\Tpereps}\|\left( U_{\rm eff}^{\varepsilon}(t) - U^{\varepsilon} (t)\right)\psi_0 \|_{L^2(\R ^n)} = 0 \, ,
\end{equation}
\end{enumerate}
where $\Tpereps$ is given in \eqref{Tper}.
\end{hypo}

\begin{remark}  For $\psi_0=\alpha_0 ^{\top}(\varepsilon \bx) \rev{\Phi_{\star}} (\bx)\in\BL $, 
Hypothesis \ref{hyp:eff}  implies 
 {\it slow envelope effective dynamics}. Indeed, let 
 \[ \rev{\alpha(T,\cdot) = \mathscr{U}_{\rm eff}(T;-i\nabla)[\alpha_0] }\, .\]
Then, using the space-time scaling $ \mathcal{S}_\varepsilon [f](\bx,t) = f(\varepsilon \bx,\varepsilon^a t)$,  we may write:
 \[   U^{\varepsilon}_{\rm eff}(t)\psi_0 = \mathcal{S}_\varepsilon\circ \mathscr{U}_{\rm eff}(t;-i\nabla)[\alpha_0]\cdot
 \mathcal{S}^{-1}_{\varepsilon}\Phi_\star(\bx) = \alpha(\varepsilon^a t,\varepsilon \bx)\cdot
\Phi_\star(\bx). \]
Equivalently, $\alpha(T,X)$ evolves under the effective Hamiltonian $H_{\rm eff}(T,-i\nabla)$, which 
generates the unitary flow $\mathcal{U}_{\rm eff}$:
\[ i\partial_T\alpha = H_{\rm eff}(T,-i\nabla)\alpha,\quad \alpha(0,X)=\alpha_0(X).\]
\end{remark}

The effective evolution operator, $U_{\rm eff}^{\varepsilon}(t)$,  naturally gives rise to an
  \begin{equation} \textrm{{\it Effective monodromy operator defined on $\BL$}:}\quad M_{\rm eff}^{\varepsilon} \equiv U_{\rm eff}^{\varepsilon}(\Tpereps); \label{eff-mono}
  \end{equation}
 for  $\psi_0=\alpha_0 ^{\top}(\varepsilon \bx) \rev{\Phi_{\star}} (\bx)\in\BL $,
 \[  (M_{\rm eff}^{\varepsilon}\psi_0)(\bx) = \rev{\left(\mathscr{U}_{\rm eff}(\Tper;-i\nabla)[\alpha_0]\right)}(\varepsilon\bx)\cdot \Phi_\star(\bx)\ .\]

\begin{hypo}[Spectrum of the effective monodromy operator]\label{hyp:spec}
For every $d_0>0$ sufficiently small there exists $g_0\in [0,\pi)$ such that \begin{equation}\label{eq:abstract_effspec}
\Spec_{\BL}(M^{\varepsilon}_{\rm eff}) \subseteq \left\{ e^{-i\nu} ~| \nu \in \left(E_{\star}\Tpereps -g_0, E_{\star}\Tpereps+ g_0 \right) \right\} \, .
\end{equation}
\end{hypo}
\begin{remark}[Notational assumption; $E_\star=0$ from here on]\label{rem:E0} In the proofs of our results below we shall, without loss of generality, by replacing $H^0$ by $H^0-E_\star$,
 set $E_\star=0$. Under this convention, \eqref{eq:abstract_effspec} in   Hypothesis~\ref{hyp:spec} simply \rev{reads as}
 \begin{equation}\label{eq:abstract_effspec1}
\Spec_{\BL}(M^{\varepsilon}_{\rm eff}) \subseteq \left\{ e^{-i\nu} ~| \nu \in \left( -g_0, g_0 \right) \right\} \, .
\end{equation}
\end{remark}

\subsection{A theorem on near-invariance of quasi-energy spectrum} 

Since the monodromy operator $M^{\varepsilon}$ is unitary (see \eqref{eq:mono_def}), $M^{\varepsilon}$ has  a spectral representation as an  integral with respect to a  projection-valued spectral measure, $\Pi^{\varepsilon}$,  which is supported  on the unit circle; see Sec.\ \ref{sec:spectral_thm}.
We now state our main theorem, which addresses Question 2.

  Denote by $(a,b)$  the arc $\{ e^{-iy} ~~| ~~ y\in (a,b)\} \subseteq S^1$.
\begin{theorem}[Near invariance]\label{thm:Peps_e3}
Consider the periodically forced Schr{\"o}dinger equation \eqref{eq:lsa_bal}. Assume that 
for some quasi-momentum / energy pair $(\bks, E_{\star})=(\bks,0)$ (see Remark \ref{rem:E0}) Hypotheses \ref{hyp:deg}--\ref{hyp:spec} are satisfied.  Let $\Peps$, defined in \eqref{eq:Pdef}, denote the $L^2(\R^n)$ projection onto Bloch modes of $H^0$ of energy and quasi-momentum in an $\varepsilon$-neighborhood of $(\bks, E_{\star})=(\bks,0)$.

Then, for every $g\in (g_0, \pi)$ there exists $\varepsilon_0>0$ such that for all $\varepsilon \in (0,\varepsilon_0)$, 
\begin{equation}\label{near-invar}
\Pi ^{\varepsilon}\left[(-g,g) \right] \circ \Peps\  = \Peps + \mathcal{O}_{\mathcal{B}(L^2)}(\varepsilon ^{n+1}).\
\end{equation}
\end{theorem}

\vspace{0.1in}

Theorem \ref{thm:Peps_e3} is a near-invariance (or stability) result for a spectral subspace associated 
 with $H^0$, the range of $\Peps$, under the perturbed dynamics $H^\varepsilon(t)$. Indeed,  let the parametric forcing term be zero, i.e., $W = 0$; Then, under Hypothesis \ref{hyp:deg}, the non-driven monodromy operator $M^{\varepsilon}_0$, restricted to the range of  $\Peps$, is given by, for any $u\in L^2(\R ^n)$:
\begin{equation}\label{eq:M0Peps}
M^{\varepsilon}_0\Peps u (\bx) = \sum_{b=b_\star}^{b_\star+N-1}\int\limits_{|\bk-\bks|<\varepsilon} \left\langle \Phi_{b}(\cdot; \bk) , u \right\rangle \Phi_{b} (\bx;\bk)  e^{-iE_{b} (\bk)\Tpereps} \, d\bk \, .\end{equation}
Suppose that the the relevant dispersion surfaces near $\bks$ are bounded by a polynomial of degree $a$, i.e., there is a constant $C> 0$ such that for $|\bk - \bks |\lesssim \varepsilon$ and for all $b_\star \leq b \leq  b_\star+N-1$ we have
  \begin{equation}
      |E_b(\bk)| \le C|\bk-\bks|^a + \mathcal{O}(\varepsilon^{a+1}) = \mathcal{O}(\varepsilon^a) \, .
  \end{equation}
Denoting the spectral measure of $M^{\varepsilon}_0$ by $\Pi^{\varepsilon}_0$, we have that for a fixed $g$ and sufficiently small $\varepsilon >0$, by inspecting the the exponents in \eqref{eq:M0Peps}
\begin{equation}\label{eq:Pi0eps_invar} 
 \Pi_0 ^{\varepsilon}\left[ (-g,g) \right]\circ \Peps = \Peps  \quad{\rm and}\quad \Pi_0 ^{\varepsilon}\left[ S^1\setminus (-g,g) \right]\circ \Peps = 0 \, .
 \end{equation}
 As discussed in Remark \ref{rem:O1effect}, it is non-trivial that a form of \eqref{eq:Pi0eps_invar} persists
  for time-periodic forcing $W\ne0$ in \eqref{eq:lsa_bal},  due to the formally order-one cumulative effect of a perturbation of size $\varepsilon^{a}$ on the time-scale $\Tpereps \sim \varepsilon^{-a}$.

 \subsection{The main result for the space of band limited wavepackets $\BL$} 

As a step toward the proof of Theorem \ref{thm:Peps_e3}, we first prove its analog, Theorem \ref{thm:main_bal}, a strict invariance property for functions in  $\BL$ (see \eqref{eq:BL_def}), a closed subspace of $L^2(\R^n)$. Since $\BL$ approximates the range of $\Peps$ (Proposition \ref{prop:BLproj}), we can then use Theorem \ref{thm:main_bal} to prove 
Theorem \ref{thm:Peps_e3}, which concerns the range of $\Peps$.


\begin{lemma}\label{lem:BL_orth}
There exists $\varepsilon_0>0$, such that for all \rev{$ \varepsilon \in (0,\varepsilon_0)$},  
$\BL$ defined in \eqref{eq:BL_def} is a closed subspace of  $L^2(\R ^n)$. Hence, $L^2(\R ^n)$ has the  decomposition
$$L^2( \R ^n) = \BL \oplus \BL ^{\perp} \, ,$$
with corresponding orthogonal projections on $L^2(\R^n)$ denoted 
\[ \textrm{${\rm Proj}_{\BL}$\quad{\rm and}\quad ${\rm Proj}_{\BL}^{\perp}= {\rm I}-
{\rm Proj}_{\BL}$.}\]
\end{lemma}
\rev{For proof, see Appendix \ref{ap:bl_orth_pf}.} $\BL$ is a very natural space with which to study the effects of time-dependent forcing. In fact, the proof of Theorem \ref{thm:Peps_e3}, follows from its analog for the space $\BL$:

\begin{theorem}[Invariance on $\BL$]\label{thm:main_bal}
Consider \eqref{eq:lsa_bal} and suppose it satisfies Hypotheses \ref{hyp:deg}--\ref{hyp:spec} at some quasi-momentum energy pair $(\bks, E_{\star}=0)$. Fix $d_0 \in (0,\pi)$ and $g>0$ such that $g \in (g_0, \pi)$. Then,  for every $g\in (g_0, \pi)$ there exists $\varepsilon_0>0$ such that for all $\varepsilon \in (0,\varepsilon_0)$ 
\begin{equation}\label{eq:invar}
\Pi^{\varepsilon} \left[(-g,g)\right] \circ {\rm Proj}_{\BL} = {\rm Proj}_{\BL}\ ,
\end{equation}

Equivalently,
\begin{equation}\label{eq:qenergy_bal_control}
\Pi^{\varepsilon} \left[ S^1 \setminus (-g,g)\right] \circ {\rm Proj}_{\BL} = 0 \, .
\end{equation}
\end{theorem}
Theorem \ref{thm:main_bal} is proved in Section \ref{sec:pf_main}.  Here, we first use it to give a proof of the main result, Theorem \ref{thm:Peps_e3} (concerning $\Peps$).

\begin{proof}[Proof of the main result, Theorem \ref{thm:Peps_e3}]
 To prove Theorem \ref{thm:Peps_e3} we shall use  Theorem \ref{thm:main_bal} above and the following Proposition, which is proved in Section \ref{sec:proj_pf}:
\begin{proposition}[$\BL$ approximates ${\rm ran}(\mathcal{P}^\varepsilon_0)$]\label{prop:BLproj}
There exists $\varepsilon_0 >0$ such that for every $0<\varepsilon<\varepsilon_0$ the following holds: for every $f\in L^2( \R ^n)$ there is a $u_{\varepsilon}[f]\in \BL$ with $d_0=1$ (see \eqref{eq:BL_def}) such that 
\begin{equation}\label{eq:Peps2BL}
 \Peps f = u_{\varepsilon}[f] + \mathcal{O}\left(\varepsilon^{n+1} \|f\|_{L^2(\R ^n)} \right) \, \, .
\end{equation}
Conversely, there exists $C>0$ such that for every
$u\in \BL$ with $d_0$ sufficiently small, then
\begin{equation}\label{eq:BL2Peps}
\left\|\left( I- \Peps \right) u \right\|_{L^2(\R ^n)} \leq C \varepsilon^{n+1} \left\|u\right\|_{L^2(\R^n )} \, .
\end{equation}
\end{proposition}
 By Proposition \ref{prop:BLproj}, $\Peps u  = u_{\rm bl} + r$, where $u_{\rm bl}\in \BL$ and $\|r\|_{L^2} = \mathcal{O}(\varepsilon ^{n+1} \|v\|_{L^2})$. Now 
\begin{align*}
\Pi^{\varepsilon} \left[ S^1 \setminus (-g,g)\right] \circ \Peps u 
&= \Pi^{\varepsilon} \left[ S^1 \setminus (-g,g)\right]u_{\rm bl} + \Pi^{\varepsilon} \left[ S^1 \setminus (-g,g)\right] r \\
&= 0 + \Pi^{\varepsilon} \left[ S^1 \setminus (-g,g)\right] r \, ,
\end{align*}
where, since $u_{\rm bl}\in \BL$, the last equality is the result of Theorem \ref{thm:main_bal}. Finally, since $\Pi^{\varepsilon} \left[ S^1 \setminus (-g,g)\right]$ is a projection, $$ \left\| \Pi^{\varepsilon} \left[ S^1 \setminus (-g,g)\right] r \right\|_{L^2(\R ^n)} \leq \left\| r \right\|_{L^2(\R ^n)} = \mathcal{O}(\varepsilon ^{n+1})\, , $$ which completes the proof.
\end{proof}

\section{Applications of the main result, Theorem \ref{thm:Peps_e3}}\label{sec:examples}

In this section we apply Theorem \ref{thm:Peps_e3} to time-periodically forced (Floquet)
Hamiltonians of the form: 
\begin{equation}\label{eq:lsa_concrete}
H^{\varepsilon}(t) = H^0 + 2i\varepsilon ^a \uA(\varepsilon^a t) \cdot \nabla \, .
\end{equation}
Here, $\uA:\R \to \R ^n$ is $\Tper$-periodic with zero mean, i.e., $\int_0^{\Tper} \uA (T) \, dT  = 0$.
A discussion, with references,  of how this class of models arises in condensed matter physics and photonics is presented in Appendix \ref{sec:phys}.

The setting of Theorem \ref{thm:Peps_e3} is a Floquet Hamiltonian, here \eqref{eq:lsa_concrete}, and a neighborhood of an energy quasi-momentum pair $(E_\star,\bk_\star)$ 
in the band structure of $H^0$, in which the class of wave-packet initial data are spectrally localized.
 Here, we characterize the local character of the band structure at $(E_\star,\bk_\star)$ by a number of parameters.
 As in Hypothesis \ref{hyp:deg}, we denote by  $N$ the multiplicity of $E_\star$. The parameter $a$ in \eqref{eq:lsa_concrete} is chosen to match the rate at which
  to energy, $E$, deviates from $E_\star$ for $|\bk-\bk_\star|$ small. 
  Table \ref{tab:examples} summarizes
  four cases of physical interest, which are discussed in the following subsections.

     \begin{table}
\centering
\scriptsize
\begin{tabular}{|c|| c | c | c| |c| c|}
 \hline \hline 
 Section  & dispersion rate $a$ &  $N$, degeneracy order & dimension $n$ &  effective equation
 \\ \hline
\ref{sec:lin}  & $1$& $1$& $n\geq 1$ & Transport \eqref{eq:tr-effH} \\ \hline
\ref{sec:dirac}  & $1$& $\substack{2 \\ \text{(+conical touching)}}$& $n=2$ & Dirac system  \eqref{eq:Dirac} \\ \hline
\ref{sec:quad}   & $2$& $1$& $n\geq 1 $ & Schr{\"o}dinger \eqref{eq:eff_sch} \\ \hline
\ref{sec:quad_deg}  & $2$& $\substack{2 \\ \text{(+quadratic touching)}}$ & $n=2$ & Schr{\"o}dinger system  \eqref{eq:coupledSchr} \\ \hline
\end{tabular}
\caption{Summary of examples discussed in Subsections \ref{sec:lin}-\ref{sec:quad_deg}. Parameters $N$ and $a$ are defined in \eqref{eq:lsa_bal} and Hypothesis \ref{hyp:deg}, respectively.}
\label{tab:examples}
\end{table}

  In what follows, $\BL$ wavepackets are always denoted by $u(\bx) = \alpha (\varepsilon \bx) ^{\top} \Phi (\bx)$, where the dimension of $\alpha$ and $\Phi$ is $N$, the degree of the degeneracy at $(\bks, E_{\star})$.
  
In the non-driven case, i.e., when $\uA = 0$, the effective/homogenized models, which govern the large time dynamics of wave-packet envelopes, 
are continuously translation-invariant PDEs of the form
$ i\partial_T \alpha = H_{\rm eff}(-i\nabla)\alpha$; see references below.  Our analysis shows that, for \eqref{eq:lsa_concrete} with $\uA \neq 0$, the dynamics of wave-packet envelopes, 
is governed by \[ i\partial_T \alpha = H_{\rm eff}(-i\nabla, T)\alpha,\]
where the non-autonomous  Hamiltonian $H_{\rm eff}(-i\nabla, T)$ is obtained from
  arising from $H_{\rm eff}(-i\nabla)$ via the formal replacement 
$
-i\nabla _X \mapsto P_{\uA}(T) \equiv -i\nabla _X + \uA(T) \, .
$
In each example below, we display $H_{\rm eff}(-i\nabla, T)$.

\subsection{ $E_\star$ simple and  $(\bks,E_\star)$ a non-critical point - ballistic transport}\label{sec:lin}
For a given Hamiltonian $H^0$, let $(\bks ,b_{\star})$ be a pair of a quasi-momentum and index $b_{\star}\in \mathbb{N}$ such that $ E_{b_{\star}} (\bks)=0$ is a simple eigenvalue of $H^0$ in $L^2_{\bks}$ with a linear dispersion relation, i.e., 
\begin{subequations}\label{eq:ballisticAss}
\begin{equation}\label{eq:ballisticAss_simple}
E_{b_{\star}-1}(\bks) < E_{\star} < E_{b_{\star}+1}(\bks) \, ,
\end{equation}
and 
\begin{equation}\label{eq:ball_slopeCond}
\bc \equiv -\grad _k E_{b_{\star}}(\bk) |_{\bk=\bks} \neq 0 \, ,
\end{equation}
where $\bc \in \R^n$ (since the dispersion surfaces are real-valued).
\end{subequations}
By continuity of the energy bands, Hypothesis \ref{hyp:deg} holds. 
The effective Hamiltonian, governing the  $\BL-$ wave-packet envelope $\alpha(X,T)$,  is given by
\begin{align}
  H_{\rm eff}(-i\nabla,T) & =  \bc \cdot P_{\uA(T)} = \bc\left( -i\nabla_X + \uA(T) \right).
  \label{eq:tr-effH}
\end{align}

 In this case, following the notations of Hypothesis \ref{hyp:eff}, $\|U^{\varepsilon}(t)-U_{\rm eff}^{\varepsilon}(t) \|_{L^2}\lesssim \varepsilon$ for $t\lesssim \varepsilon^{-1}$.  The proof of this statement follows closely that of \rev{of the case of a double conical degeneracy (Dirac point; see Section \ref{sec:dirac}),} which is presented in detail in \cite{sagiv2022effective}. We include a formal derivation of \eqref{eq:tr-effH} in Section \ref{sec:transport}. 

To verify Hypothesis \ref{hyp:spec}, we apply the Fourier  transform  (in the $X$ variable) $\mathcal{F}[\alpha(T,X)](\bxi) = \hat{\alpha}(T;\bxi)$, and get the  family of ODE initial value problems, parametrized by $\xi\in \R ^n $:
$$i\partial_T \hat{\alpha}(T;\bxi) = \bc \cdot \left(\bxi + \uA(T) \right)\hat{\alpha}(T;\bxi),\  \hat{\alpha}(0;\bxi)=\hat{\alpha}_0(\bxi),$$
for which the solution is
\begin{equation}\label{eq:transport_fourier}
\hat{\alpha}(T;\bxi) =  \exp \left[-i \bc \cdot \left(\bxi T + \uh (T) \right) \right]\ \hat{\alpha}_0 (\bxi) \, , \qquad \uh (T) \equiv \int\limits_{0}^{T} \uA (T') \, dT' \, .
\end{equation}
Hence, for a fixed $d_0>0$ and $u\in \BL$, 
\begin{equation}\label{eq:eff_Mon}
\Meff u \equiv U_{\rm eff}^{\varepsilon}(\Tper \varepsilon^{-1})u = \rev{(2\pi)^{-n/2}} \varepsilon ^{\frac{n}{2}} \int\limits_{|\bxi|\leq d_0} e^{i\bxi \cdot \varepsilon \bx} \hat{\alpha}_0(\bxi) e^{-i \bc \cdot \bxi \Tper  } \Phi(x) \, d\bxi   \, ,
\end{equation}
\rev{where we recall that, by assumption, $\uh (\Tper )=0$.}And so, by choosing wavepackets supported on the ball $|\bxi|<d_0$, 
$$\sigma (M_{\rm eff}^{\varepsilon}) ~~ \text{on}~~\BL =  \{e^{iy}~~| ~~ y\in [-d_0\Tper|\bc|, d_0\Tper |\bc| ]~\}\, , $$
which verifies Hypothesis \ref{hyp:spec}.

\subsection{ $E_\star$ of multiplicity two; conical touching of dispersion surfaces at $(\bks,E_\star)$, a Dirac point}\label{sec:dirac}

An example which plays an important role in the modeling of  two-dimensional materials such as  graphene is the case where $H^0=-\Delta + V(\bx)$, where $V$ is a honeycomb lattice potential, {\it i.e.} $V$  has the symmetries of a honeycomb tiling of $\R^2$.  (A one-dimensional variant of such potentials,  {\it dimer potentials}, was studied in  \cite{FLW-MAMS:17} 
and the following discussion can adapted to this setting as well.)  For generic honeycomb lattice potentials, conical degeneracies (Dirac points) occur in the band structure at pairs $(\bks, E_\star)$, where $\bks$ is any vertex (high symmetry quasimomentum) of the hexagonal Brillouin zone  \cite{FW:12}. In a neighborhood of a Dirac point one has two consecutive dispersion surfaces, $E_-(\bks)\le E_+(\bks)$, satisfying
\begin{equation}
E_\pm(\bk) = E_\star \pm v_{\rm D}|\bk - \bks| + \mathcal{O}\left( |\bk-\bks|^2 \right) \, ,\  v_{\rm D}>0.
\label{eq:conical}\end{equation}
The slope of the cone,  $v_{\rm D}$, is referred to as the Dirac or Fermi velocity. From \eqref{eq:ball_slopeCond} we see that Hypothesis \ref{hyp:deg} is satisfied.

Hypothesis \ref{hyp:eff} is also satisfied for this class of equations. \rev{Indeed, \cite[Theoem 3.2]{sagiv2022effective} shows that, with scaling parameter $a=1$ (see \eqref{eq:conical}), the effective envelope dynamics of \eqref{eq:lsa_concrete} for data in
 $\BL$ are governed by a driven Dirac Hamiltonian:
\begin{align}\label{eq:Dirac}
 H_{\rm eff}(T,-i\nabla_X)&=v_{\rm D}\left( \sigma_1, \sigma_2 \right)\cdot \left(P_{\uA_1(T)},P_{\uA_2(T)}\right)
\, ,
\end{align}
and that $\|(U_{\rm eff}(\Tpereps)-U^{\varepsilon}(\Tpereps))f\|_{L^2 (\R^n)}\lesssim \varepsilon \|f\|_{L^2 (\R^n)}$ for data $f\in \BL$.}

Finally, \rev{Hypothesis \ref{hyp:spec} is satisfied by \cite[Proposition 3.5]{sagiv2022effective}.} There,  we apply the Fourier transform to the effective Dirac equation above and find  the eigenvalues of its monodromy operator at Fourier momentum: $\xi = (0,0)^{\top}$. By continuity with respect to $\xi$, one can find $g_0$ for a sufficiently small $d_0$ to satisfy Hypothesis \ref{hyp:spec}. \rev{In the semi-classical regime (which is different from the present paper), wave-packet dynamics near a conical singularity have been studied extensively, see e.g., \cite{cances2021coherent, watson2017wavepackets, watson2018wavepackets}}

\subsection{ $E_\star$ simple, $(\bks,E_\star)$ a non-degenerate  critical (quadratic) point of a band}\label{sec:quad}

Suppose $(\bks,E_\star)$ is such that $E_\star$ is a 
simple $L^2_{\bks}$-eigenvalue of $H^0$ and $\bks$ is a non-degenerate critical point of the band dispersion function $E_b$:  $\grad_{\bk} E_b(\bks) = \vec 0 $ and $\det D_{\bk}^2E_b(\bks)\ne0$. 
Then, $a=2$, and the envelope dynamics for are given by an driven effective Schr{\"o}dinger-type Hamiltonian:
\begin{align}\label{eq:eff_sch}
H_{\rm eff}(T,-i\nabla_X) &= P _{\uA(T)}\cdot \frac12 D^2_\bk E(\bks) P_{\uA(T)}  \,  .
\end{align}
The validity on time scales of order $\varepsilon^{-2}$, and therefore Hypothesis \ref{hyp:eff}  follows
along the lines of  \cite{allaire2005homogenization} or \cite{sagiv2022effective}.

Note that such quadratic points may occur at spectral band edges, in which case the Hessian $D_{\bk}^2E(\bks)$ is positive  or negative definite or at $(\bks,E_\star)$; or where $E_\star$ is interior to a spectral band, in which case the Hessian $D_{\bk}^2E(\bks)$ might have an indefinite signature. Finally,  a similar homogenization argument can be carried in the case where  $\grad_{\bk} E_b(\bks) \neq \vec 0 $ and $D_{\bk}^2E(\bks)$ non-degenerate. In this case  one gets a Schr{\"o}dinger equation on the time-scales of $\varepsilon^{-2}$, with a drift term on the time-scale of $\varepsilon^{-1}$; see, for example, \cite{allaire2005homogenization}.

\subsection{$E_\star$ of multiplicity two; quadratic touching of two dispersion surfaces at $(\bks,E_\star)$}\label{sec:quad_deg} 
 Consider a two-dimensional Hamiltonian $H^0 = -\Delta +V(\bx)$ where the  potential $V$ which is periodic with respect to the lattice $\Lambda=\mathbb Z^2$, real-valued, even, and invariant 
under a $\pi/2-$ rotation. We can take the Brillouin zone, $\mathcal B$, to be a square, centered at the origin in $\R^2_\bk$. The vertices of $\mathcal B$ are {\it high-symmetry quasi-momenta}.   In  \cite{keller2018spectral} it is proved that the band structure of $H^0$ has consecutive band dispersion surfaces which touch quadratically over the vertices of $\mathcal B$ at an eigenvalue with a two-fold degenerate eigenvalue.

Hence, we consider \eqref{eq:lsa_concrete} with $a=2$ for $\BL$ data near these high-symmetry $\bks$-points. The effective envelope dynamics of $\BL$ data can be shown, in a manner analogous to the derivation in \cite{keller2020erratum},  to be governed by the  matrix-Schr{\"o}dinger effective Hamiltonian:
\begin{subequations}\label{eq:coupledSchr}
 \begin{equation}
 H_{\rm eff}\left(-i\nabla,T\right) = \alpha \left(P_{_{A_1(T)}}^2+P_{_{A_2(T)}}^2\right)\sigma _0 +
 \tilde{\gamma}\left(P_{_{A_1(T)}}^2 - P_{A,2}^2 \right)\sigma_2+2\beta P_{_{A_1(T)}} P_{_{A_2(T)}}\sigma_1 \, ,
 \end{equation}
 where
 \begin{equation}
 P_{_{A_j(T)}} \equiv -i\partial_{X_j} + A_j (T)\, ,  ~~ j=1,2 \, .
\end{equation}
\end{subequations}
Here,  the coefficients $\alpha, \tilde{\gamma}, \beta\in\R$  and can be expressed as $L^2(\R/\mathbb Z^2)$ inner products  involving a basis for the $2-$dimensional $L^2_{\bks}-{\rm kernel}$ of $(H^0-E_\star)$; see \cite{keller2020erratum}, and $\sigma_0,\sigma_1, \sigma_2$ are Pauli matrices. As in the case of the effective Dirac equation (Section \ref{sec:dirac}), Hypothesis \ref{hyp:spec} is verified as follows: {\it (i)} Fourier-transforming \eqref{eq:coupledSchr} yields a system of $2$ linear and time-periodic system of (Floquet) ODES, which is parametrized by $\bxi$. {\it (ii)} Since this matrix defining this system of ODEs  has trace equal to $0$ and is continuous in $\bxi$, the Floquet multipliers $e^{\pm i\mu (\bxi)}$ are continuous functions of $\bxi$ on unit circle. {\it (iii)} Hence, for $\BL$ data with a fixed band-width $d_0>0$ (see Definition \ref{def:BL}), there exists a continuous function $g_0 (d_0)$ such that the $\BL$ data of $M^{\varepsilon}_{\rm eff}$ is contained in the arc $(-g_0(d_0), g_0(d_0))$.

\section{Proof of Theorem \ref{thm:main_bal}}\label{sec:pf_main}

Let us first recall the following centering lemma for unitary operators. Intuitively, it says that if a unitary operator acts on a function which is spectrally localized, it is approximately the same as acting as a multiplication operator. We proved a weaker version of this lemma in \cite{sagiv2022effective}, and include the proof here for completeness.
\begin{lemma}\label{lem:center}
Let $\mathcal{I}\subset S^1$ such that $\Pi^{\varepsilon}(\mathcal{I})u=u$ and let $e^{-i\nu_0}\in \mathcal{I}$ be the mid-point of the arch $\mathcal{I}$. Then 
$$M^{\varepsilon}u= e^{-i\nu_0} u + \eta \, , \qquad {\rm where}\, , \qquad \|\eta\|_{L^2(\R ^n)} \leq 2\sin  \left( \frac{\left| \mathcal{I}\right|_{S^1}}{4} \right) \cdot \|u\| \, ,$$
where $\left| \mathcal{I}\right|_{S^1}$ is the arclength of $\mathcal{I}$.
\end{lemma}
\begin{proof}
Let $z_0=e^{-i\nu_0}$. Then
\begin{align*}
M^{\varepsilon} u &= \int\limits_{ \mathcal{I}} z d\Pi^{\varepsilon} (z) u \\
&= \int\limits_{\mathcal{I}} \left( z_0-z_0-z\right) d\Pi^{\varepsilon} (z) u \\
&= z_0 u + \eta \, , \quad {\rm where} \quad \eta \equiv \int\limits_{\mathcal{I}} \left( z-z_0\right) d\Pi^{\varepsilon} (z) u  \, .
\end{align*}
Since $z_0= e^{-i\nu_0}$, we only need to bound $\|\eta\|_{L^2}$. By the spectral theorem (see Sec.\ \ref{sec:spectral_thm}), we have that
\begin{align*}
\|\eta \|_{L^2 (\R ^n)}^2  &= \int\limits_{\mathcal{I}} |z-z_0|^2 \langle d\Pi^{\varepsilon}(z)u, u \rangle_{L^2(\R ^n)} \\
&\leq \max\limits_{z\in \mathcal{I}}|z-z_0|^2 \cdot \int\limits_{S^1} \langle d\Pi^{\varepsilon}(z)u, u \rangle_{L^2(\R ^n)} \\ 
&=\max\limits_{e^{-i\nu}\in \mathcal{I}}|e^{-i\nu} - e^{-i\nu _0}|^2 \cdot  \|u\|^2 \\
&\leq 4\sin ^2 \left( \frac{\left| \mathcal{I}\right|_{S^1}}{4} \right) \cdot \|u\|^2 \, ,
\end{align*}
where we have used the expression for the arc length: $|e^{i\beta} - e^{i\beta^\prime}|_{S^1} = 2\sin (|\beta - \beta^\prime|/2)$ for any $\beta, \beta ' \in [0,2\pi)$ with $|\beta - \beta^\prime|\leq \pi$, combined with the fact that $\nu_0$ is the mid-point of $\mathcal{I}$.
\end{proof}
\begin{proof}[Proof of Theorem \ref{thm:main_bal} ]
To prove \eqref{eq:qenergy_bal_control}, let $v\in \BL$ and let $$v' \equiv \Pi^{\varepsilon} \left[ S^1\setminus  (-g,g) \right] v \, ,$$
for $g\in (g_0, \pi)$, where $g_0$ is defined in Hypothesis \ref{hyp:spec}. We will now show that $v'=0$. Lemma \ref{lem:center} implies that, since $\pi$ is the midpoint of the arch  $\mathcal{I}=S^1 \setminus (-g,g)$,
$$M^{\varepsilon} v' = e^{-i\pi}v' + \eta_{v'} = -v' + \eta_{v'} \, , \qquad {\rm where} \quad  
\|\eta_{v'} \|_{L^2(\R ^n)} \leq  2\sin \left( \frac{\pi -g}{2} \right) \cdot \|v'\|_{L^2 (\R ^n)} \, .$$
Hence
\begin{align*}
 \left\| \left( M^{\varepsilon} - \Meff \right)v' \right\|_{L^2(\R^n)} &= \left\| \left( -{\rm Id} - \Meff \right)v' + \eta_{v'} \right\|_{L^2(\R^n)} \\
 &\geq  \left\| \left( -{\rm Id} - \Meff \right)v'  \right\|_{L^2(\R ^n)}-\left\| \eta_{v'} \right\|_{L^2(\R^n)} \\
  &\geq  \left\| \left( -{\rm Id} - \Meff \right)v'  \right\|_{L^2(\R ^n)}-2\sin \left( \frac{\pi -g}{2} \right) \cdot \left\| v' \right\|_{L^2(\R^n)} \, . \numberthis \label{eq:M-Meff_2_Id-M-Meff}
\end{align*}  

To bound $\| \left( -{\rm Id} - \Meff \right)v'  \|_{L^2(\R ^n)}$ from below, we will prove the following lemma:

\begin{lemma}\label{lem:lbd_einu0Meff}
For any $g\in (g_0, \pi)$ there exists $\varepsilon_0>0$ such that for all  $\varepsilon \in (0, \varepsilon_0)$, $\nu_0 \in (g_0, \pi]$,\footnote{An analogous formula holds if $\nu_0 \in [\pi, 2\pi-g_0)$.} and any $f \in \BL$,
$$\left\| \left( e^{-i\nu_0} - \Meff\right) f \right\|_{L^2( \R  ^n)} \geq  2 \sin  \left( \frac{\nu_0 - g_0}{2} \right) \left\| f \right\|_{L^2(\R^n)} \, .$$
\end{lemma}
Let us first use lemma \ref{lem:lbd_einu0Meff} with $\nu_0 = \pi$ to prove the main result, Theorem \ref{thm:main_bal}, and then return to its proof. Combined with \eqref{eq:M-Meff_2_Id-M-Meff}, we have that
   \begin{align*}
    \left\| \left( M^{\varepsilon} - \Meff \right)v' \right\|_{L^2(\R^n)}&\geq \cdots 
    \\
    &\geq   \rev{\frac{2}{{\rm vol}(\Omega)^{1/2}}}\sin \left( \frac{\pi -g_0}{2} \right)\left\|v'  \right\|_{L^2(\R ^n)}-2\sin \left( \frac{\pi -g}{2} \right) \cdot \left\| v' \right\|_{L^2(\R^n)}\\
        &\geq 2 \left[ \rev{\frac{1}{{\rm vol}(\Omega)^{1/2}}}\sin \left( \frac{\pi -g_0}{2} \right) - \sin \left( \frac{\pi -g}{2} \right)\right] \cdot \left\|v'  \right\|_{L^2(\R ^n)} \, .
\end{align*}
On the other hand, since $v'\in \BL$, Hypothesis \ref{hyp:eff} regarding the effective dynamics provides an upper bound on $ \| \left( M^{\varepsilon} - \Meff \right)v' \|_{L^2(\R^n)}$. 
When combined this yields that
$$ 2 \left[\rev{\frac{1}{{\rm vol}(\Omega)^{1/2}}} \sin \left( \frac{\pi -g_0}{2} \right) - \sin \left( \frac{\pi -g}{2} \right)\right] \cdot \left\|v'  \right\|_{L^2(\R ^n)} \leq \left\| \left( M^{\varepsilon} - \Meff \right)v' \right\|_{L^2(\R^n)} \leq o (\varepsilon) \cdot \left\| v' \right \|_{L^2 (\R ^n)} \, .$$
Since $\pi >g>g_0$, the difference on the left-hand side above is always positive. Therefore, for sufficiently small $\varepsilon>0$, the above inequality is only possible if $v' =~0$.
\end{proof}

\begin{proof}[Proof of Lemma \ref{lem:lbd_einu0Meff} ]
By the explicit form of $\Meff$ given in Hypothesis \ref{hyp:eff}, we can write for every $f\in \BL$,
\begin{align*}
 \left( e^{-i\nu_0}- \Meff \right)f  & 
=  \rev{(2\pi)^{-n/2}}\varepsilon ^{\frac{n}{2}} \int\limits_{|\bxi|\leq d_0} e^{-i\bxi \cdot \varepsilon \bx} \, \left[ \left( e^{-i\nu_0}{\rm Id}- \hat{M}_{\rm eff}^{\varepsilon}(\bxi)  \right)  \hat{\alpha}_0(\bxi) \right] ^{\top}  \Phi (\bx) \, d\bxi  \\ 
&= \varepsilon ^{\frac{n}{2}} \gamma  (\varepsilon \bx) ^{\top} \Phi (\bx)  \, . \numberthis  \label{eq:meff_gamma_pre}
\end{align*}
where 
\begin{equation}
\gamma(X) \equiv \rev{(2\pi)^{-n/2}}\int\limits_{|\bxi|\leq d_0} e^{-i\bxi \cdot X} \, \left[ \left( e^{-i\nu_0}{\rm Id}- M_{\rm eff}^{\varepsilon}(\bxi) \right) \hat{\alpha}_0 (\bxi) \right]  \, d\bxi
\end{equation}
Next, we recall the following averaging lemma:

\begin{lemma}\label{lem:avg} Let $q\in L^2(\R ^n)$ \rev{with $\int_{\R^n}q(X)\, dX <\infty$} and such that ${\rm supp}(\hat{q})\subseteq B(0,d)$ for some $d>0$, and let $p\in L^2(\Omega)$ be $\Lambda$-periodic. Then, there exists $\varepsilon_0 >0$ which depends on $d$, such that for any fixed $0<\varepsilon < \varepsilon_0$, 
\begin{equation}\label{eq:avrg}
\int\limits_{\R^n} p(\bx)q(\varepsilon \bx) \, d\bx = \rev{\frac{\varepsilon^{-n}}{{\rm vol}(\Omega)}} \left( \int\limits_{\Omega} p(\bx) \, d\bx \right) \cdot \left( \int\limits_{\R ^n} q(X) \, dX \right) \, .
\end{equation}
\end{lemma}
\rev{For proof, see Appendix \ref{ap:avg_pf}.} Applying Lemma \ref{lem:avg} to \eqref{eq:meff_gamma_pre} yields, using the orthonormality of $\Phi_b, \ldots , \Phi_{b+N-1}$ (for brevity, set $\rev{b_{\star}}=1$ without loss of generality)
\begin{align*}
\left\|\varepsilon ^{\frac{n}{2}} \gamma (\varepsilon \bx) \Phi (\bx) \right\|_{L^2(\R^n)} ^2 &= \varepsilon ^n \int\limits_{\R ^n} \left|\gamma (\varepsilon \bx)^{\top}   \Phi (\bx) \right| ^2 \, d\bx \\
&= \varepsilon ^n \int\limits_{\R ^n} \sum\limits_{j, m=1}^N \gamma _j(\varepsilon \bx)   \Phi _j(\bx) \bar{\gamma} _m(\varepsilon \bx)   \bar{\Phi} _m(\bx) \, d\bx \\
&=\varepsilon^n \rev{\frac{\varepsilon^{-n}}{{\rm vol}(\Omega)}}\sum\limits_{j,m=1}^N \langle \gamma_m , \gamma_j \rangle_{L^2(\R ^N)} \cdot \langle \Phi_m, \Phi_j \rangle_{L^2_{\bks}} \\
&=\rev{\frac{1}{{\rm vol}(\Omega)}}\sum\limits_{j=1}^N \|\gamma _j \|^2_{L^2 (\R ^n)} =\rev{\frac{1}{{\rm vol}(\Omega)}} \|\gamma\|_{L^2(\R ^n;\C^N)}^2 \, ,
\end{align*}
where in applying Lemma \ref{lem:avg}, we used the fact that, while the support of the Fourier transform of $\gamma _j \bar{\gamma}_m$ might not be $B(0,d_0)$, it is still \rev{compact, since it is included in $B(0,2d_0)$}. 

Hence, to prove Lemma \ref{lem:lbd_einu0Meff}, we need to bound the norm of $ \|\gamma\|_{L^2(\R ^n;\C^N)}$ from below. We now note that for every $\bxi \in \R ^n$, the Fourier-transformed monodromy $\hat{M}_{\rm eff}^{\varepsilon}(\bxi)$ is an $N\times N$ unitary matrix (where $N$ is the degree of the degeneracy in Hypothesis \ref{hyp:deg}).  Let $P(\bxi)$ be the unitary matrix which \rev{diagonalizes} the monodromy, i.e.,  $$\hat{M}_{\rm eff}^{\varepsilon}(\bxi) = P(\bxi) D(\bxi) P^* (\bxi) \, , \qquad D(\bxi)_{\ell,j} = e^{-i\nu_j(\bxi)} \delta _{j,\ell} \, , \quad 1\leq j,\ell \leq N \, .$$
Hence, using \rev{Plancherel} theorem and the orthogonality of $P(\bxi)$, we have that
\begin{align*}
\left\| \gamma\right\|_{L^2(\R^n)}^2 &=  \left\| \rev{(2\pi)^{-n/2}}\int\limits_{|\bxi|\leq d_0} e^{-i\bxi \cdot X} \, \left[P(\bxi) \left( e^{-i\nu_0}{\rm Id}- D(\bxi)  \right)P^*(\bxi)  \hat{\alpha}_0 (\bxi) \right]  \, d\bxi \right\|_{L^2 (\R ^n_X; \C ^N)}^2  \\
&= \left\|P(\bxi) \left( e^{-i\nu_0}{\rm Id}- D(\bxi)  \right)P^*(\bxi)  \hat{\alpha}_0 (\bxi)    \right\|_{L^2 (\R ^n_{\bxi} ;\C^N)}^2 \\
&= \left\|\left( e^{-i\nu_0}{\rm Id}- D(\bxi)  \right)P^*(\bxi)  \hat{\alpha}_0 (\bxi)  \right\|_{L^2 (\R ^n_{\bxi} ;\C^N)}^2 \\
&= \sum\limits_{j=1}^N \left\|\left( e^{-i\nu_0} - e^{-i\nu_j (\bxi)} \right) \left( P^*(\bxi)  \hat{\alpha}_0 (\bxi)\right)_j    \right\|_{L^2 (\R ^n_{\bxi} )}^2 \\
&= \sum\limits_{j=1}^N \int\limits_{|\bxi|\leq d_0} \left| e^{-i\nu_0} - e^{-i\nu_j (\bxi)} \right|^2 \left|\left( P^*(\bxi)  \hat{\alpha}_0 (\bxi)\right)_j\right|^2  \, d\bxi  \\
 &\geq  \sum\limits_{j=1}^N \min\limits_{|\bxi '|\leq d_0}\left| e^{-i\nu_0} - e^{-i\nu_j (\bxi ')} \right|^2 \cdot \left\| \left( P^*(\bxi)  \hat{\alpha}_0 (\bxi)\right)_j    \right\|_{L^2 (\R ^n_{\bxi} )}^2  \\
&\geq \min\limits_{|\bxi|\leq d_0} \min\limits_{1\leq j\leq N} \left| e^{-i\nu_0}- e^{-i\nu_j (\bxi) \Tper} \right| ^2 \cdot \left\| \alpha _0 \right\|_{L^2(\R^n ;\C ^N)}^2  \\
&\geq 4 \sin ^2  \left( \frac{\nu_0 - g_0}{2} \right)\cdot \left\| \alpha _0 \right\|_{L^2(\R^n ;\C ^N)}^2  \, ,
\end{align*}
where the last inequality is derived from the arc-length formula between two angles, as well as from Hypothesis \ref{hyp:spec} on the spectrum of $M_{\rm eff}^{\varepsilon}$.


 
\end{proof}

\subsection{Proof of Proposition \ref{prop:BLproj}}\label{sec:proj_pf}
We note here that the proof of Proposition \ref{prop:BLproj} is very similar to that which appears in \cite{sagiv2022effective}. However, due to many changes in the notation and change in dimensionality, we include it here for completeness.
\subsubsection{From projections to wavepackets; proof of \eqref{eq:Peps2BL}}

Let $\varepsilon >0$ be taken sufficiently small, and let $f\in L^2(\mathbb{R}^n)$. Express $H^0$ acting in $ L^2(\mathbb{R}^2)$ as a direct integral $H^0 = \int^\oplus_{\mathcal{B}} H^0_k \, d\bk$,  where  $H^0_\bk $ denotes the operator $H = -\Delta +V$ acting in $ L^2_\bk$.   Then, taking $E_b (\bks)=0$ without loss of generality, we can rewrite \eqref{eq:Pdef}
\begin{align*}
\Peps &= \int\limits_{\mathcal{B}} \, d\bk \  \chi \left(\frac{|\bk-\bks|}{\varepsilon}<1\right) {\rm Proj}_{L^2_\bk}(|H _\bk^0 |<L\varepsilon) \, f \numberthis \label{eq:proj_peps_int} \\
&= \int\limits_{\mathcal{B}} \, d\bk \,  \chi \left(\frac{|\bk-\bks|}{\varepsilon}<1 \right) \left[\frac{1}{2\pi i}\oint\limits_{|\zeta |=2L\varepsilon} \,   (\zeta I   - H_\bk)^{-1}\ d\zeta \right] \, f  \, , \numberthis \label{eq:proj_cont} \\
\end{align*}
\noindent
 where the the factor $2$ in the $2L\varepsilon$ radius in the contour integral is not necessarily sharp.  In order to expand for $\bk$ near $\bks$,  we next express the operators $H_\bk$ in terms of operators which acts in the fixed space $L^2_{\bks}$.
Note that  $H_\bk = e^{i\bk \cdot  \bx} H(\bk) e^{-i\bk \cdot  \bx}$, where $H(\bk)\equiv -(\nabla +i\bk)^2 + V$  acts in $L^2(\R^n/\Lambda)$.
Furthermore, $ (\zeta I  - H_\bk)^{-1}=e^{i\bk \cdot  \bx}(\zeta I- H(\bk) )^{-1} e^{-i\bk \cdot  \bx}  $.

Substitution into \eqref{eq:proj_cont} yields
\begin{align*}
\cdots
&=  \int\limits_{\mathcal{B}} \, d\bk \,  e^{i\bk \cdot  \bx}\chi \left(\frac{|\bk-\bks|}{\varepsilon}<1\right) \left[\frac{1}{2\pi i}\oint\limits_{|\zeta |=2L\varepsilon}\  (\zeta I - H(\bk))^{-1}\ \, d\zeta \right] \, e^{-i\bk \cdot  \bx}f  \\
\qquad &=  \int\limits_{\mathcal{B}} \, d\kappa  \,  e^{i(\bks+\kappa )\cdot \bx}\chi \left(\frac{\kappa}{\varepsilon} <1\right) \left[\frac{1}{2\pi i}\oint\limits_{|\zeta |=2L\varepsilon} \,  (\zeta I   - H(\bks+\kappa ))^{-1}\ d\zeta \,  \right] \, e^{-i(\bks+\kappa)\cdot \bx}f   \, .
\end{align*}
The contour integral inside the square brackets is smooth $L^2(\R^2/\Lambda)$-valued function of $\kappa$, and so by Taylor expansion:
\begin{align}
\cdots  = & \int\limits_{\mathcal{B}} \, d\kappa  \,  e^{i(\bks+\kappa )\cdot \bx}\chi \left(\frac{|\kappa|}{\varepsilon} <1\right) \left[\frac{1}{2\pi i}\oint\limits_{|\zeta |=2L\varepsilon} \,  (\zeta I  - H(\bks))^{-1}  \, d\zeta \right] \, e^{-i(\bks+\kappa)\cdot \bx}f  \nonumber\\
&+  \int\limits_{\mathcal{B}} \, \chi \left(\frac{\kappa}{\varepsilon} <1\right) \ \kappa\ \mathcal{\rm Error}[f;\kappa]\ d\kappa  \, .
\label{texp}\end{align}
The last term in \eqref{texp} is linear in $f$ and easily seen to be bounded in $L^2(\R^2)$ by $\varepsilon^{n+1} \|f\|_{L^2}$ since the domain of integration
 is over a disc of radius $\varepsilon$.  
 
The dominant term in \eqref{texp} may be re-expressed as
 \begin{align*}
 & \int\limits_{\mathcal{B}} \, d\kappa  \,  \chi \left(\frac{|\kappa|}{\varepsilon} <1\right) 
e^{i\kappa\cdot \bx}\left[\frac{1}{2\pi i}\oint\limits_{|\zeta |=2L\varepsilon} \,  e^{i\bK\cdot \bx}(\zeta I  - H(\bks))^{-1}e^{-i\bK\cdot \bx}  \, d\zeta \right] \, e^{-i\kappa\cdot \bx}f (\bx)\nonumber\\
\quad  = & \int\limits_{\mathcal{B}} \, d\kappa  \,  \chi \left(\frac{|\kappa|}{\varepsilon} <1\right) 
e^{i\kappa\cdot \bx}\left[\frac{1}{2\pi i}\oint\limits_{|\zeta - E_D|=2L\varepsilon} \,  
(\zeta I  - H_{\bks})^{-1}  \, d\zeta \right] \, e^{-i\kappa\cdot \bx}f(\bx) \nonumber\\
= &\int\limits_{\mathcal{B}} \, d\kappa  \,  \chi \left(\frac{|\kappa|}{\varepsilon} <1\right) 
e^{i\kappa\cdot \bx}\ 
 {\rm Proj}_{ L^2_{\bks} }(|H _{\bks} |<2L\varepsilon) \ e^{-i\kappa\cdot \bx}f(\bx)\\
&= \int\limits_{\mathcal{B}}\ d\kappa\  \chi\left(\frac{|\kappa|}{\varepsilon} < a\right)\ e^{i\kappa \cdot \bx}\ \rev{\Phi_{\star}^{\top}}(\bx; \bks) \left[\int\limits_{\mathbb{R}^n} \, d\by \, \overline{\rev{\Phi_{\star}} (\by;\bks)}f(\by)e^{-i\kappa \cdot \by} \right]   \nonumber\\
&= \rev{\Phi^{\top}_{\star}}(\bx;\bks) \int\limits_{\mathbb{R}^n} \, d\by \, \overline{\rev{\Phi_{\star}}(\by;\bks)} (\by)f(\by)  \left[ \int\limits_{\mathcal{B}} \, d\kappa \,  \chi\left(\frac{|\kappa|}{\varepsilon} < 1\right)e^{i \kappa\cdot  (\bx-\by)} \right]  \nonumber \\
&= \rev{\Phi_{\star}}^{\top} (\bx;\bks)\int\limits_{\mathbb{R}^n} \, d\by \, \overline{\rev{\Phi_{\star}} (\by;\bK)}f(\by)  \left[ \int\limits_{\mathbb{R}^n} \, d\kappa \,  \chi\left(\frac{|\kappa|}{\varepsilon} < 1 \right)e^{i\kappa \cdot (\bx-\by)} \right]  \nonumber 
\end{align*}
In all, we have that
 \begin{align}
 \Peps f  &=  u_{\varepsilon}[f] + \mathcal{O}_{L^2(\R^n)}\left(\varepsilon^{n+1} \|f\|_{L^2}\right) \, ,
 \label{texp1} \end{align}
 where
  \begin{align}
 u_\varepsilon[f](\bx) &\equiv  \Phi^{\top}(\bx;\bks) \beta_\varepsilon [f](\bx),\quad {\rm and}  \nonumber\\
 \beta_\varepsilon[f](\bx) &\equiv \left[\ \left( \overline{\rev{\Phi_{\star}}(\cdot;\bks)} f \right) \ast \mathcal{F}_{\bxi}^{-1}\left[ \chi \left( \frac{|\bxi|}{\varepsilon}<1 \right)\right]\ \right](\bx) \, ,
 \label{u-beta}
\end{align}
where $\mathcal{F}^{-1}[g](\bxi)$ denotes the inverse Fourier transform and $\ast$ denotes convolution. We next show that $u_\varepsilon \in \BL$ with $d_0=1$ by showing that 
$\mathcal{F}[\beta_\varepsilon [f] ] \in\chi (|\bxi|< \varepsilon)L^2(\R ^n) $.
 Indeed by the convolution rule,  \[
\mathcal{F}[ \beta_\varepsilon [f] ](\bxi) = \mathcal{F}\left[\left( \overline{\Phi(\cdot ;\bks)} f \right) \ast \mathcal{F}^{-1}\left[ \chi \left( \frac{|\bxi|}{\varepsilon}<1 \right)\right]\right] = \mathcal{F}[\overline{\rev{\Phi_{\star}}(\cdot ;\bks)} f](\bxi)  \, \,\chi\left(\frac{|\bxi|}{\varepsilon} <1 \right) ,
\]
which is supported in $\{|\bxi|<\varepsilon a\}$. This completes the proof of \eqref{eq:Peps2BL}.
\begin{remark}
This proof shows that, more generally, if the definition of $\Peps$ would have been changed to a \rev{projection} onto the disc $|\bk-\bks|<a \varepsilon$ with $a\neq 1$, then the Proposition would have carried through with a different value of $d_0$.
\end{remark}

\subsubsection{From wavepackets to projections; proof of \eqref{eq:BL2Peps}}\label{sec:bl2peps}

 Consider
$u(\bx)\in \BL$ for some $d_0 \in (0,1)$ and $\varepsilon > 0$ sufficiently small, then by definition of \eqref{eq:BL_def}, there exists $\alpha_{\varepsilon}\in L^2 (\R^n ;\C ^N)$ such that  
$$ u(\bx)=~\rev{\Phi_{\star}}^{\top}(\bx;\bks)\alpha_\varepsilon (\bx) \, ,\quad \textrm{where}\quad 
\mathcal{F}\left[ \alpha_\varepsilon \right](\bxi) = \chi\left(\frac{|\bxi|}{\varepsilon}<d_0\right)\mathcal{F}\left[\alpha_\varepsilon\right] (\bxi) \, . $$
 On the other hand, by \eqref{texp1}, for any $\BL$ function and $\varepsilon>0$ sufficiently small, there exists a function $\gamma_{\varepsilon}$ such that
\begin{equation}\label{eq:Proj_BL}
\Peps u =  \rev{\Phi^{\top}_{\star}}(\bx, \bks)\gamma_{\varepsilon}[u] (\bx) + \mathcal{O}(\varepsilon^{n+1}\|u\|_{L^2(\mathbb{R}^n)}) \, .
\end{equation}
To prove \eqref{eq:BL2Peps} it suffices to show that 
$\gamma_{\varepsilon}(\bx)= \alpha_\varepsilon(\bx)$. 
Substitution of 
 $u=\rev{\Phi_{\star}}^{\top}(\bx;\bks)\alpha_{\varepsilon}(\bx)$ into \eqref{u-beta} yields
$$ \gamma_{\varepsilon}[u] = 
\left(\overline{\rev{\Phi_{\star}}(\cdot;\bks)}\rev{\Phi_{\star}}^{\top}(\cdot;\bks)\alpha_\varepsilon \right) \ast\mathcal{F}^{-1}\left[\chi\left(\frac{|\bxi|}{\varepsilon}<1 \right)\right] \,  . $$
We next  compute the Fourier transform of $\gamma_{\varepsilon}[u]$. For $\rev{b_{\star}}\leq j < b+N$,
\begin{align}
    \mathcal{F}\left[ \gamma_{\varepsilon}[g]\right]_j &=\mathcal{F}[\overline{\rev{\Phi_{\star}}(\bx; \bks)}\rev{\Phi_{\star}}^{\top}(\bx; \bks)\alpha _{\varepsilon}( \bx) ]_j ~  \chi\left(\frac{|\bxi|}{\varepsilon}<1 \right) \nonumber\\
    &= \sum\limits_{\ell =b}^{N+b-1} \mathcal{F}[\overline{\Phi_j (\bx;\bks)}\Phi_{\ell}(\bx; \bks) \alpha_{\varepsilon, \ell}(\bx)] ~ \chi\left(\frac{|\bxi|}{\varepsilon}<1 \right) \label{eq:gammaFour_terms} 
\end{align}
Consider the expression being summed in \eqref{eq:gammaFour_terms}. Since $\Phi_j(\bx;\bks)=e^{i\bks\cdot\bx}\phi_j(\bx;\bks)$ with
 $\phi_j(\bx;\bks)\in L^2(\R^n/\Lambda)$ periodic, we have
  \begin{align*}
 p_{j,\ell}(\bx) \equiv \overline{\Phi_j(\bx, \bks)}\Phi_{\ell}(\bx, \bks) =  \overline{\phi_j(\bx, \bks)}\phi_{\ell}(\bx, \bks) \in L^2(\R^n/\Lambda) \, .
 \end{align*}
We expand $p_{j,\ell}(\bx)$ for each $j,\ell =1,2$ in a  Fourier series with respect to the lattice $\Lambda$: for every $g\in L^2 (\R^n / \Lambda)$, $$g(\bx) = \sum_{\bn \in \Lambda^{*}} \hat{g}(\bn) e^{i\bn \cdot \bx} \, , \qquad \hat{g}(\bn) \equiv \int\limits_{\Omega} e^{-i\bn \cdot \bx} g(\bx) \, d\bx \, .$$ Substituting the Fourier series into \eqref{eq:gammaFour_terms} yields
\begin{align*}
\mathcal{F}[\gamma_{\varepsilon}[u]]_j &= \sum\limits_{\ell = n}^{N+b-1}\mathcal{F} \left[ \sum\limits_{\bn \in \Lambda^*} \hat{p}_{j,\ell}(\bn)e^{i\bn \cdot x} \alpha_{\varepsilon,\ell }(\bx) \right]~  \, \, \chi\left(\frac{|\bxi|}{\varepsilon}<1 \right)\\
&= \sum\limits_{\ell = n}^{N+b-1}\sum\limits_{\bn \in \Lambda^*} \hat{p}_{j,\ell}(\bn) \mathcal{F} \left[  e^{i\bn \cdot x} \alpha_{\varepsilon, \ell}(\bx) \right] ~ \chi\left(\frac{|\bxi|}{\varepsilon}<1 \right)\\
&= \sum\limits_{\ell = n}^{N+b-1}\sum\limits_{\bn \in \Lambda^*} \hat{p}_{j,\ell}(\bn) \hat{\alpha}_{\varepsilon, \ell} \left(\bxi -  \bn\right) ~ \chi\left(\frac{|\bxi|}{\varepsilon}<1 \right) \numberthis \label{eq:gamma_sum} \\
\end{align*}
Note that by definition, $\hat{\alpha_{\varepsilon}}$ has compact support in the disc of radius $\varepsilon a $ around the origin. In the expansion above in \eqref{eq:gamma_sum}, for $\varepsilon>0$ sufficiently small, the only term that does not vanish only if (i) $|\bxi-\bn|<\varepsilon $ (with $\bn\in\Lambda^*$) and 
 (ii) due to the $\chi(|\bxi|<\varepsilon)$ term, if $|\bxi|<\varepsilon $. Hence, the only non-zero term in \eqref{eq:gamma_sum} arises from the lattice point $\bn = \vec 0$. Then, by definition of the Fourier coefficient $\hat{p}_{j,\ell} (\vec 0)$ and the orthogonality of the different $\Phi_j$'s, we have that
\begin{align*}
\mathcal{F}[\gamma_{\varepsilon}[u]]_j &= \sum\limits_{\ell = b}^{N+b-1} \hat{p}_{j,\ell}(\vec 0) \hat{\alpha}_{\varepsilon, \ell} \left(\bxi\right) ~ \cdot\chi\left(\frac{|\bxi|}{\varepsilon}<1 \right) \\
&= \sum\limits_{\ell = n}^{N+b-1} \int_\Omega \overline{\Phi_j(\by;\bks)}\Phi_l(\by;\bks) \,  d\by\
 \hat{\alpha}_{\varepsilon, \ell}\left(\bxi\right)~ \cdot\chi\left(\frac{|\bxi|}{\varepsilon}<1 \right)\\
 &= \hat{\alpha}_{\varepsilon, j}\left(\bxi\right)~ \cdot\chi\left(\frac{|\bxi|}{\varepsilon}<1 \right) = {\rm Vol}(\Omega)^{-1} \hat{\alpha}_{\varepsilon, j}\left(\bxi\right) \,.
 \end{align*}
 
 Summarizing, we have $\gamma _{\varepsilon}[u] =\alpha_\varepsilon $ 
  and  and so substitution into \eqref{eq:Proj_BL}  yields
   \begin{equation}\label{eq:Proj_BL1}
{\rm Proj}_{L^2(\mathbb{R}^2)} (|H^0-E_{\star}|\leq 
\varepsilon)u = u(\bx) + \mathcal{O}(\varepsilon^{n+1}\|u\|_{L^2(\mathbb{R}^2)}) \, .
\end{equation}
This is equivalent to \eqref{eq:BL2Peps}. The proof of Proposition \ref{prop:BLproj} is now complete.

\section{Effective transport dynamics}\label{sec:transport}
Consider \eqref{eq:lsa_concrete} with $a=1$ and initial data of the form
\begin{equation}\label{eq:wp_data}
\psi_0 (\bx) = \varepsilon ^{\frac{n}{2}} \alpha_0(\varepsilon \bx) \Phi_b (\bx; \bks) \, , \qquad \alpha_0 \in H^s(\R ^n) \, ,
\end{equation}
for sufficiently high $s>0$, and where the $\varepsilon ^{\frac{n}{2}}$ factor keeps the overall norm of $\psi_0$ independent of $\varepsilon$. Initial data in $\BL$ is then a sub-class of \eqref{eq:wp_data}. In this subsection, we formally derive the effective transport equation and its propagator $U_{\rm eff}^{\varepsilon}$, as given in \eqref{eq:tr-effH}. The proof of its validity follows closely that of \cite[Theorem 3.2]{sagiv2022effective} and \cite[Theorem 5.1]{hameedi2022radiative}, and is therefore omitted from this manuscript.

To construct a solution, we assume separation of scales, with slow time variables
$$T\equiv \varepsilon t \, , \quad X \equiv \varepsilon \bx \, , $$
 and introduce the expansion $$\psi(t,x) = \psi^0 (t,x) + \varepsilon \psi^1 (t,x) + \cdots $$
 where for every $j\geq 0$ $$\psi^j (t,x) = \Psi ^j (t,x, T,X) |_{T=\varepsilon t, X = \varepsilon x } \, .$$ 
By expanding 
$$\partial_t \mapsto \partial_t +\varepsilon T \, , \quad \vec \nabla \mapsto \vec \nabla _{\bx} + \varepsilon \vec \nabla _X  \, , \quad \Delta \mapsto \Delta _{\bx} + 2\varepsilon\vec \nabla _{\bx} \cdot \vec \nabla _X +\Delta _X \, .$$
and substituting into \eqref{eq:lsa_bal}, we solve for each power of $\varepsilon$. 

\textbf{Order $\varepsilon^0$.}
$$(i\partial_t - H^0)\Psi^0 = 0 \, ,\qquad \Psi(t=0,T=0,x,X) = \alpha_0(X)\Phi (\bx) \, ,$$
and so $\Psi^0(t,T,\bx,X) = \alpha_0(X)\Phi (\bx)$.

\textbf{Order $\varepsilon^1$.}
$$(i\partial_t -H^0)\Psi^1 = \left(-i\partial_T +2\vec \nabla _{\bx} \cdot \vec \nabla _X +2i\uA (T)\cdot \vec \nabla _{\bx}\right) \Psi^0 \, .$$
To invert $(i\partial_t - H^0)$ in $L^2_{\bks}$ and solve for $\Psi^1$, we need to verify that the right hand side is $L^2_{\bk}$ orthogonal to the kernel, i.e., to $\Phi = \Phi_b(\cdot;\bk)$ (from here on, we suppress the $\bks$ and $b$ dependence for brevity).

 Here, it is useful to note that \eqref{eq:ballisticAss} is equivalent to a statement on the Bloch mode $\Phi_b(x;\bks)$:
\begin{lemma}\label{lem:c_slope}
Given \eqref{eq:ballisticAss}, then 
\begin{equation}
\label{eq:ballisticAss_ip}
\langle \Phi_b (\cdot;\bks) , 2\vec \nabla \Phi_b (\cdot;\bks) \rangle_{L^2_{\bks}} = i\bc \neq \vec 0 \, .
\end{equation}
\end{lemma}
Combining Lemma \ref{lem:c_slope} and normalizing $\langle \Phi, \Phi\rangle_{L^2_{\bks}} = 1$, we get the desired result
$$
i\partial_T \alpha(T,X) = i\bc \cdot \left(\vec \nabla _X +i\uA(T) \right)\alpha \, . 
$$
\begin{proof}[Proof of Lemma \ref{lem:c_slope}]
By definition, $\Phi$ satisfies 
$$ H \Phi (\bx; \bk) = E(\bk) \Phi \, , \qquad \Phi \in L^2_{\bk} \, .$$
Write  
\begin{equation}\label{eq:pkdef}
\Phi(\bx;\bk) = e^{i\bk \cdot \bx} p_{\bk}(\bx) \, , \qquad p_{\bk} \in L^2_{\rm per}(\Omega) \, ,
\end{equation}
which transforms the TISE to 
\begin{equation}\label{eq:tise_k} \hat{H}({\bk}) p_{\bk} = E(\bk) p_{\bk} \, , \qquad \hat{H}(\bk) \equiv H -2i\bk \cdot \grad_x + |\bk|^2 \, .  
\end{equation}
We now take $\grad _k$ on both sides of \eqref{eq:tise_k}.
By noting that 
\begin{equation*}
\grad _k \hat{H}(\bk) =- 2i \grad _x +2 \bk \, ,
\end{equation*}
we get that
\begin{equation*} 
 \hat{H}(\bk) \grad _k p_{\bk} - 2i \grad_x p_{\bk} + 2\bk p_{\bk} = \grad _k E (\bk) p_{\bk} + E(\bk) \grad _k p_{\bk} \, .  
\end{equation*}
We rearrange some of the terms and take the inner product from the left with $p_{\bk}$
$$ \langle p_{\bk} ,p_{\bk} \grad_k  E(\bk) \rangle = \left\langle p_{\bk}, \left( \hat{H}(\bk) - E(\bk) \right) \grad _k p_{\bk} \right\rangle - \langle p_{\bk}, 2i\grad _x p_{\bk} \rangle + \langle p_{\bk} , 2\bk p_{\bk} \rangle \, .$$
Here we note that, by definition $\|\phi (\cdot; \bk) \|_2 = \|p_{\bk} \|_2 =1 $. Moreover, since $\hat{H}(\bk)-E(\bk)$ is self-adjoint, combined with \eqref{eq:tise_k}, the first inner product on the right-hand side vanishes. By differentiating \eqref{eq:pkdef}, we get
\begin{align*}
\grad _k E(\bk) &= \langle p_{\bk}, 2(-i\grad _x +\bk)p_{\bk} \rangle \\
&= \left\langle \Phi (\bx ;\bk)e^{-i\bk \cdot \bx}, -2i\grad _x \left( \Phi(\bx;\bk)\right) e^{-i\bk \cdot \bx} \right\rangle\\
&= -\langle \Phi(\cdot ; \bk), 2i \grad _x \Phi(\cdot ; \bk) \rangle \, .
\end{align*}
\end{proof}

\section*{Data Availability Statement}
Data sharing not applicable to this article as no datasets were generated or analyzed during the current study.

\appendix
\section{Physical interpretations of the model}\label{sec:phys}
%
  An example of physical interest is the case of \eqref{eq:lsa_concrete}, i.e.,  \eqref{eq:lsa_bal} with $W(T,-i\nabla) = -2i\uA(T)\cdot \nabla$. 
Note here that \eqref{eq:lsa_concrete} can be transformed to an equivalent  ``magnetic'' form 
 \[ i\partial_t\psi = ( i\grad +\varepsilon^a \uA (\varepsilon^a t)  )^2 \psi + V\psi ,\]
 where $\uA$ is a vector potential. 
  This class of PDEs arises in physical settings, such as:
\begin{enumerate}[label=(\alph*)]
\item The modeling of time periodic conductors (e.g., graphene), excited by a time-varying electric field  \cite{perez2014floquet, wang2013observation}. Here, $H^0=-\Delta+V$ is a single-electron Hamiltonian for graphene and time-dependence in $H^\varepsilon(t)$ models the excitation of the graphene sheet by an external electric field with no magnetic field (by Maxwell's equations, since $\uA$ is constant in space, see e.g., \cite{krieger1986time}).

\item For $n=1,2$,
the propagation of light in a periodic array of  helically coiled optical fiber waveguides \cite{bellec2017non, jurgensen2021quantized, ozawa2019topological, Rechtsman-etal:13}. Here, the Schr{\"o}dinger equation describes the propagation in the time-like longitudinal direction of a continuous-wave (CW) laser beam propagating through a hexagonal or triangular transverse array of optical fiber waveguides. Beginning with Maxwell's equations, under the nearly monochromatic and paraxial approximations, one obtains $i\psi_z (z,\bx)= H^0 \psi$ for the longitudinal evolution of the slowly varying envelope of the classical electric field. Suppose the fibers are longitudinally coiled. Then, in a rotating coordinate frame, we obtain \eqref{eq:lsa_bal} where the time-periodic perturbation, $\uA$, captures effect of periodic coiling.
\end{enumerate}

\appendix
\section{Auxilary proofs}
\subsection{Proof of Lemma \ref{lem:BL_orth}}\label{ap:bl_orth_pf}
\rev{It suffices to prove that $\BL$ is a closed subspace of $L^2(\mathbb{R}^n)$. Let $({\alpha}_m^{\top}(\varepsilon \cdot){\Phi})_{m=1}^{\infty}\subset L^2 (\R^n) $ be a sequence  converging to some $u\in L^2(\mathbb{R}^n)$, and then let us show that $u\in \BL$, too.
To do so, we first prove that  $(\alpha_m)_{m=1}^{\infty}$ is a Cauchy sequence in $L^2(\mathbb{R}^n;\mathbb{C}^N)$.\footnote{Recall that $N\geq 1$ is the degree of the degeneracy at $(\bks,E_{\star})$, see Hypothesis \ref{hyp:deg}.}  This is a consequence of the  averaging lemma, Lemma \ref{lem:avg}. Indeed, choose for any $1\leq j,l,\leq N$, and every two indexes $m_1, m_2\geq 1$,
 \[ q_{j,l}(X)= \overline{(\alpha_{j,m_1} - \alpha_{j,m_2})(X)} (\alpha_{l,m_1} - \alpha_{l,m_2})(X) \, ,\qquad  p_{j,l}(\bx)= \overline{\Phi_j(\bx)}\Phi_l(\bx) \, .\] 
  Since $\{\Phi_1,\ldots , \Phi_N\}$ is an orthonormal set in $L^2(\Omega)$,  for all $\varepsilon>0$ sufficiently small, and every two indexes $m_1, m_2$:
  \begin{align*}
  \int_{\mathbb{R}^n}\big| \left(\alpha_{m_1} (\varepsilon \bx) - \alpha_{m_2} (\varepsilon \bx)\right)^{\top}\Phi(\bx)\big|^2\ d\bx &= \frac{\varepsilon^{-n}}{{\rm vol}(\Omega)} \sum\limits_{j,l=1}^N \left(\int\limits_{\Omega} p_{j,l}(\bx) \, d\bx\right) \cdot \left(~ \int\limits_{\mathbb{R}^n} q_{j,l}(X) \, dX\right) \\
  &= \frac{\varepsilon^{-n}}{{\rm vol}(\Omega)} \sum\limits_{j,l=1}^N \delta _{j,l} \cdot \left( ~\int\limits_{\mathbb{R}^n} q_{j,l}(X) \, dX\right) \\
  &=\frac{\varepsilon^{-n}}{{\rm vol}(\Omega)}\|\alpha_{m_1} - \alpha_{m_2} \|_2^2 . 
   \end{align*}
Since the left hand side tends to zero as $m_1$ and $m_2$ tend to infinity, so does the right hand side. Therefore, 
 $(\alpha_m)_{m=1}^{\infty}$ is a  Cauchy sequence and converges in $L^2(\R^n ; \C ^N)$ to some $\alpha_\star$. Hence, 
 the sequence $\{ \varepsilon \alpha_m^{\top}(\varepsilon \cdot) \Phi\}$ converges to
  $\varepsilon \alpha_\star^{\top}(\varepsilon \cdot) \Phi$ . }

  \rev{Finally, while the limit does have the correct form, for $u$ to be in $\BL$ we still need to show that ${\rm supp}(\hat\alpha_\star)\subset \{|\bxi|\le d_0\}$ to complete the proof.  Indeed, since   $(\alpha_m)_{m=1}^{\infty}$ converges in $L^2 $,  by the Plancherel identity. Since all $\hat{\alpha}_m$ are complactly supported on the same compact set, it mus be that the sequence
  $\hat\alpha_m(\bxi)$ converges almost everywhere in $\xi$, up to a subsequence. Furthermore, for all $m$ we have $\hat\alpha_m(\bxi)\equiv0$ for $|\xi|>d_0$, so  we conclude $\hat\alpha_\star(\bxi)\equiv0$ for $|\bxi|>d_0$.}

\subsection{Proof of Lemma \ref{lem:avg}}\label{ap:avg_pf}
  
\rev{Since the fundamental cell of the lattice $\Omega$ tiles the plane, we partition $\mathbb{R}^n$ with respect to the lattice, i.e., $\mathbb{R}^n = \bigcup_{\bm\in \Lambda} (\Omega +\bm)$.
Therefore 
\begin{align*}
\int\limits_{\mathbb{R}^2}p(\bx) 	q(\varepsilon \bx) \, d\bx &= \sum\limits_{\bm \in \Lambda} \, \, \int\limits_{\Omega + \bm}  p(\bx)q(\varepsilon \bx) \, d\bx \\
(\text{change of variables} ~~ \bx=\by+\bm)\qquad  \qquad  &=\sum\limits_{\bm \in \Lambda}\, \, \int\limits_{\Omega} p(\by) q(\varepsilon(\by + \bm)) \, d\by \\
&=\int\limits_{\Omega} p(\by)\left[ \sum\limits_{\bm \in \Lambda} q(\varepsilon(\by + \bm))\right] \, d\by \, . \numberthis \label{eq:gpre_pois}
\end{align*}
Using a generalization of Poisson summation formula for general lattices \cite{stein2016introduction}, then
$$\sum\limits_{\bm \in \Lambda} q(\by + \bm) = \frac{1}{{\rm vol}(\Omega)}\sum\limits_{\bn \in \Lambda} \hat{q}(\bn)e^{2\pi i  \by\cdot \bn} \, ,$$
where $\hat{q}$ is the Fourier transform of $q$. Since in $\R ^n$ we have that $\widehat{q(\varepsilon \cdot)}(\bxi) = \varepsilon^{-n} \hat{q}(\varepsilon ^{-1}\bxi )$, then 
\begin{align*}
\sum\limits_{\bm \in \Lambda} q(\varepsilon(\by + \bm)) &= \frac{1}{{\rm vol}(\Omega)}  \sum\limits_{\bn \in \Lambda} \varepsilon^{-n}\hat{q}\left(\frac{\bn}{\varepsilon}\right)e^{2\pi i  \by\cdot \bn} \, .
\end{align*}
Now, since $q$ is band-limited, then for sufficiently small $\varepsilon$ all of the terms in the last sum vanish but $\bn=(0,\ldots ,0)$. Therefore, and since $q$ is integrable,
$$\sum\limits_{\bm \in \Lambda} q(\varepsilon(\by + \bm))= \frac{1}{{\rm vol}(\Omega)}\varepsilon^{-n}\hat{q}(0) = \frac{\varepsilon^{-n}}{{\rm vol}(\Omega)}\int\limits_{\mathbb{R}^2} q(\bx) \,d\bx \, , $$ which when substituted into \eqref{eq:gpre_pois}, yields
\begin{align*}
\int\limits_{\mathbb{R}^2}p(\bx) q(\varepsilon \bx) \, d\bx = \cdots &= \int\limits_{\Omega} p(\by)\left[ \frac{\varepsilon^{-n}}{{\rm vol}(\Omega)}\int\limits_{\mathbb{R}^2} q(\bx)\, d\bx \right] \, d\by \\
&= \frac{\varepsilon^{-n}}{{\rm vol}(\Omega)}\left( \int\limits_{\Omega} p(\by)\,d\by \right) \cdot \left( \int\limits_{\mathbb{R}^2} q(\bx) \, d\bx \right) \, . 
\end{align*}
}

\bibliographystyle{abbrv}
\bibliography{floquetBib}

\begin{thebibliography}{10}

\bibitem{allaire2005homogenization}
G.~Allaire and A.~Piatnitski.
\newblock Homogenization of the schr{\"o}dinger equation and effective mass theorems.
\newblock {\em Communications in mathematical physics}, 258(1):1--22, 2005.

\bibitem{avron1987adiabatic}
J.~Avron, R.~Seiler, and L.~Yaffe.
\newblock Adiabatic theorems and applications to the quantum hall effect.
\newblock {\em Communications in Mathematical Physics}, 110(1):33--49, 1987.

\bibitem{bal2021multiscale}
G.~Bal and D.~Massatt.
\newblock Multiscale invariants of {F}loquet topological insulators.
\newblock {\em arXiv preprint arXiv:2101.06330}, 2021.

\bibitem{bambusi2017reduce}
D.~Bambusi.
\newblock Reducibility of 1-d {S}chr{\"o}dinger equation with time quasiperiodic unbounded perturbations. {I}.
\newblock {\em Trans. Amer. Math. Soc.}, 370:1823--1865, 2017.

\bibitem{bambusi2001time}
D.~Bambusi and S.~Grafi.
\newblock Time quasi-periodic unbounded perturbations of {S}chr{\"o}dinger operators and {KAM} methods.
\newblock {\em Comm. Math. Phys.}, 219:465--480, 2001.

\bibitem{bellec2017non}
M.~Bellec, C.~Michel, H.~Zhang, S.~Tzortzakis, and P.~Delplace.
\newblock Non-diffracting states in one-dimensional floquet photonic topological insulators.
\newblock {\em EPL (Europhysics Letters)}, 119(1):14003, 2017.

\bibitem{cances2021coherent}
E.~Canc{\`e}s, C.~Fermanian~Kammerer, A.~Levitt, and S.~Siraj-Dine.
\newblock Coherent electronic transport in periodic crystals.
\newblock {\em Annales Henri Poincar{\'e}}, 22(8):2643--2690, 2021.

\bibitem{davies1978open}
E.~Davies and H.~Spohn.
\newblock Open quantum systems with time-dependent hamiltonians and their linear response.
\newblock {\em Journal of Statistical Physics}, 19(5):511--523, 1978.

\bibitem{davies1996spectral}
E.~B. Davies.
\newblock {\em Spectral theory and differential operators}, volume~42.
\newblock Cambridge University Press, 1996.

\bibitem{de2022locobatic}
W.~De~Roeck, A.~Elgart, and M.~Fraas.
\newblock Locobatic theorem for disordered media and validity of linear response.
\newblock {\em arXiv preprint arXiv:2203.03786}, 2022.

\bibitem{eliasson2008reducibility}
H.~Eliasson and S.~Kuksin.
\newblock On reducibility of {S}chr{\"o}dinger equations with quasiperiodic in time potentials.
\newblock {\em Comm. Math. Phys.}, 289:125--135, 2008.

\bibitem{FLW-MAMS:17}
C.~L. Fefferman, J.~P. Lee-Thorp, and M.~I. Weinstein.
\newblock Topologically protected states in one-dimensional systems.
\newblock {\em Memoirs of the American Mathematical Society}, 247(1173), 2017.

\bibitem{FW:12}
C.~L. Fefferman and M.~I. Weinstein.
\newblock Honeycomb lattice potentials and {D}irac points.
\newblock {\em J. Amer. Math. Soc.}, 25(4):1169--1220, 2012.

\bibitem{FW:14}
C.~L. Fefferman and M.~I. Weinstein.
\newblock Wave packets in honeycomb lattice structures and two-dimensional {D}irac equations.
\newblock {\em Commun. Math. Phys.}, 326:251--286, 2014.

\bibitem{feola2020reducibility}
R.~Feola, B.~Gr{\'e}bert, and T.~Nguyen.
\newblock Reducibility of {S}chr{\"o}dinger equation on a zoll manifold with unbounded potential.
\newblock {\em Journal of Mathematical Physics}, 61(7):071501, 2020.

\bibitem{garrido1964generalized}
L.~Garrido.
\newblock Generalized adiabatic invariance.
\newblock {\em Journal of Mathematical Physics}, 5(3):355--362, 1964.

\bibitem{GRW:21}
J.~Guglielmon, M.~C. Rechtsman, and M.~I. Weinstein.
\newblock Landau levels in strained two-dimensional photonic crystals.
\newblock {\em Phys. Rev. A}, 103:013505, 2021.

\bibitem{hagedorn2002elementary}
G.~A. Hagedorn and A.~Joye.
\newblock Elementary exponential error estimates for the adiabatic approximation.
\newblock {\em Journal of mathematical analysis and applications}, 267(1):235--246, 2002.

\bibitem{hall2013quantum}
B.~C. Hall.
\newblock {\em Quantum theory for mathematicians}.
\newblock Springer, 2013.

\bibitem{hameedi2022radiative}
S.~N. Hameedi, A.~Sagiv, and M.~I. Weinstein.
\newblock Radiative decay of edge states in floquet media.
\newblock {\em arXiv preprint arXiv:2201.11219}, 2022.

\bibitem{HK:10}
M.~Z. Hasan and C.~L. Kane.
\newblock Colloquium: {T}opological {I}nsulators.
\newblock {\em Reviews of Modern Physics}, 82:3045, 2010.

\bibitem{henheik2022adiabatic}
J.~Henheik and S.~Teufel.
\newblock Adiabatic theorem in the thermodynamic limit: Systems with a gap in the bulk.
\newblock In {\em Forum of Mathematics, Sigma}, volume~10, page~e4. Cambridge University Press, 2022.

\bibitem{hoefer2011defect}
M.~A. Hoefer and M.~I. Weinstein.
\newblock Defect modes and homogenization of periodic {S}chr{\"o}dinger operators.
\newblock {\em SIAM journal on mathematical analysis}, 43(2):971--996, 2011.

\bibitem{howland1989floquetII}
J.~Howland.
\newblock Floquet operators with singular spectrum. {II}.
\newblock {\em Ann. de l'I.H.P. Sec. A}, 50(3):325--334, 1989.

\bibitem{howland1974stationary}
J.~S. Howland.
\newblock Stationary scattering theory for time-dependent hamiltonians.
\newblock {\em Mathematische Annalen}, 207(4):315--335, 1974.

\bibitem{joye2022adiabatic}
A.~Joye.
\newblock Adiabatic lindbladian evolution with small dissipators.
\newblock {\em Communications in Mathematical Physics}, 391(1):223--267, 2022.

\bibitem{jurgensen2021quantized}
M.~J{\"u}rgensen, S.~Mukherjee, and M.~C. Rechtsman.
\newblock Quantized nonlinear thouless pumping.
\newblock {\em Nature}, 596(7870):63--67, 2021.

\bibitem{kato1950adiabatic}
T.~Kato.
\newblock On the adiabatic theorem of quantum mechanics.
\newblock {\em Journal of the Physical Society of Japan}, 5(6):435--439, 1950.

\bibitem{keller2018spectral}
R.~T. Keller, J.~L. Marzuola, B.~Osting, and M.~I. Weinstein.
\newblock Spectral band degeneracies of $\pi /2$-rotationally invariant periodic schrodinger operators.
\newblock {\em Multiscale Modeling \& Simulation}, 16(4):1684--1731, 2018.

\bibitem{keller2020erratum}
R.~T. Keller, J.~L. Marzuola, B.~Osting, and M.~I. Weinstein.
\newblock Erratum: Spectral band degeneracies of $\pi/2$-rotationally invariant periodic schr{\"o}dinger operators.
\newblock {\em Multiscale Modeling \& Simulation}, 18(3):1371--1373, 2020.

\bibitem{kraisler2025time}
J.~Kraisler, A.~Sagiv, and M.~I. Weinstein.
\newblock On the time-decay of solutions arising from periodically forced dirac hamiltonians.
\newblock {\em Journal of Differential Equations}, 440:113449, 2025.

\bibitem{krieger1986time}
J.~Krieger and G.~Iafrate.
\newblock Time evolution of {B}loch electrons in a homogeneous electric field.
\newblock {\em Physical Review B}, 33(8):5494, 1986.

\bibitem{Kuchment:12}
P.~Kuchment.
\newblock {\em Floquet Theory for Partial Differential Equations}, volume~60.
\newblock Birkhauser, Basel, 2012.

\bibitem{Kuchment:16}
P.~Kuchment.
\newblock An overview of periodic elliptic operators.
\newblock {\em Bull. Amer. Math. Soc.}, 53(3):343--414, 2016.

\bibitem{montalto2021linear}
R.~Montalto and M.~Procesi.
\newblock Linear {S}chr{\"o}dinger equation with an almost periodic potential.
\newblock {\em SIAM Journal on Mathematical Analysis}, 53(1):386--434, 2021.

\bibitem{nenciu1980adiabatic}
G.~Nenciu.
\newblock On the adiabatic theorem of quantum mechanics.
\newblock {\em Journal of Physics A: Mathematical and General}, 13(2):L15, 1980.

\bibitem{nenciu1993linear}
G.~Nenciu.
\newblock Linear adiabatic theory. exponential estimates.
\newblock {\em Communications in mathematical physics}, 152(3):479--496, 1993.

\bibitem{ozawa2019topological}
T.~Ozawa, H.~M. Price, A.~Amo, N.~Goldman, M.~Hafezi, L.~Lu, M.~C. Rechtsman, D.~Schuster, J.~Simon, and O.~Zilberberg.
\newblock Topological photonics.
\newblock {\em Reviews of Modern Physics}, 91(1):015006, 2019.

\bibitem{perez2014floquet}
P.~M. Perez-Piskunow, G.~Usaj, C.~A. Balseiro, and L.~F. Torres.
\newblock {F}loquet chiral edge states in graphene.
\newblock {\em Physical Review B}, 89(12):121401, 2014.

\bibitem{Rechtsman-etal:13}
M.~C. Rechtsman, Y.~Plotnik, J.~M. Zeuner, D.~Song, Z.~Chen, A.~Szameit, and M.~Segev.
\newblock Topological creation and destruction of edge states in photonic graphene.
\newblock {\em Phys. Rev. Lett.}, 111:103901, 2013.

\bibitem{RS4}
M.~Reed and B.~Simon.
\newblock {\em Methods of Modern Mathematical Physics: Analysis of Operators, Volume IV}.
\newblock Academic Press, 1978.

\bibitem{sagiv2022effective}
A.~Sagiv and M.~I. Weinstein.
\newblock Effective gaps in continuous floquet hamiltonians.
\newblock {\em SIAM Journal on Mathematical Analysis}, 54(1):986--1021, 2022.

\bibitem{stein2016introduction}
E.~M. Stein and G.~Weiss.
\newblock {\em Introduction to Fourier Analysis on Euclidean Spaces (PMS-32), Volume 32}.
\newblock Princeton university press, 2016.

\bibitem{taylor2013partial}
M.~Taylor.
\newblock {\em Partial differential equations II: Qualitative studies of linear equations}, volume 116.
\newblock Springer Science \& Business Media, 2013.

\bibitem{wang2013observation}
Y.~Wang, H.~Steinberg, P.~Jarillo-Herrero, and N.~Gedik.
\newblock Observation of {F}loquet-{B}loch states on the surface of a topological insulator.
\newblock {\em Science}, 342(6157):453--457, 2013.

\bibitem{watson2018wavepackets}
A.~Watson and M.~I. Weinstein.
\newblock Wavepackets in inhomogeneous periodic media: propagation through a one-dimensional band crossing.
\newblock {\em Communications in Mathematical Physics}, 363(2):655--698, 2018.

\bibitem{watson2017wavepackets}
A.~B. Watson, J.~Lu, and M.~I. Weinstein.
\newblock Wavepackets in inhomogeneous periodic media: Effective particle-field dynamics and berry curvature.
\newblock {\em Journal of Mathematical Physics}, 58(2), 2017.

\bibitem{xue2022topological}
H.~Xue, Y.~Yang, and B.~Zhang.
\newblock Topological acoustics.
\newblock {\em Nature Reviews Materials}, 7(12):974--990, 2022.

\bibitem{yajima1987existence}
K.~Yajima.
\newblock Existence of solutions for schr{\"o}dinger evolution equations.
\newblock {\em Communications in Mathematical Physics}, 110(3):415--426, 1987.

\end{thebibliography}
\end{document}